\documentclass[a4paper, 11pt]{amsart}

\usepackage{lmodern}
\usepackage[english]{babel}
\usepackage[utf8]{inputenc}
\usepackage{amssymb}
\usepackage{amsfonts}
\usepackage{amsthm}
\usepackage{amsmath}

\usepackage{mathrsfs, soul}

\usepackage{hyperref}
\usepackage{enumitem}
\usepackage{bm}
\usepackage{tikz-qtree}

\usepackage[margin=1.1in]{geometry}
\setlist[itemize]{leftmargin=*}

\usepackage[normalem]{ulem}

\theoremstyle{plain}
\newtheorem{theorem}{Theorem}[section]
\newtheorem{corollary}[theorem]{Corollary}

\newtheorem{lemma}[theorem]{Lemma}
\newtheorem{proposition}[theorem]{Proposition}
\newtheorem{conjecture}[theorem]{Conjecture}

\theoremstyle{definition}

\theoremstyle{remark}
\newtheorem*{remark}{Remark}

\numberwithin{equation}{section}

\newcommand\numberthis{\stepcounter{equation}\tag{\theequation}}

\renewcommand{\setminus}{\smallsetminus}
\newcommand{\ssum}[1]{\sum_{\substack{#1}}}

\newcommand{\e}{{\rm e}}
\newcommand{\df}{\mathop{}\!\mathrm{d}}
\newcommand{\eps}{{\varepsilon}}

\providecommand{\C}{}
\renewcommand{\C}{{\mathbb C}}
\newcommand{\E}{{\mathbb E}}

\newcommand{\K}{{\mathbb K}}
\newcommand{\N}{{\mathbb N}}
\newcommand{\Q}{{\mathbb Q}}
\renewcommand{\P}{{\mathbb P}}
\newcommand{\R}{{\mathbb R}}

\newcommand{\Z}{{\mathbb Z}}
\newcommand{\1}{{\mathbf 1}}

\newcommand{\CC}{{\mathcal C}}
\newcommand{\fd}{{\mathfrak d}}
\newcommand{\II}{{\mathfrak I}}
\newcommand{\gS}{{\mathfrak S}}

\newcommand{\vphi}{{\bm \phi}}
\newcommand{\vpsi}{{\bm \psi}}
\newcommand{\vchi}{{\bm \chi}}
\newcommand{\vt}{{\bm t}}

\newcommand{\cE}{{\mathcal E}}

\newcommand{\cH}{{\mathcal H}}
\newcommand{\I}{{\mathcal I}}

\newcommand{\cR}{{\mathscr R}}

\newcommand{\J}{{\mathbb J}}

\let\H\relax
\DeclareMathOperator{\H}{{\mathbb H}}

\DeclareMathOperator{\Id}{Id}

\DeclareMathOperator{\Res}{Res}
\DeclareMathOperator{\srd}{srd}

\DeclareMathOperator{\denom}{denom}

\DeclareMathOperator{\sgn}{sgn}
\DeclareMathOperator{\Hol}{{\mathit H}}

\renewcommand{\tilde}{\widetilde}
\renewcommand{\bar}{\overline}
\renewcommand{\hat}{\widehat}

\newcommand{\HN}[2]{\norm{#2}_{(#1)}}

\newcommand{\abs}[1]{{\left| {#1} \right|}}
\newcommand{\norm}[1]{{\left\| {#1} \right\|}}
\newcommand{\floor}[1]{{\left\lfloor {#1} \right\rfloor}}
\newcommand{\syf}[1]{\left(\!\left(#1\right)\!\right)}
\newcommand{\pmat}[1]{\begin{pmatrix}#1\end{pmatrix}}
\newcommand{\smat}[1]{\left(\begin{smallmatrix}#1\end{smallmatrix}\right)}

\renewcommand{\mod}[1]{\ ({\rm mod\ }#1)}

\renewcommand\Re{\operatorname{Re}}
\renewcommand\Im{\operatorname{Im}}

\newcommand\syb[1]{\langle{#1}\rangle}

\newcommand{\nvt}{\norm{\vt}}
\newcommand{\pdt}{{\mathfrak T}}

\numberwithin{equation}{section}

\title[Limit laws for rational continued fractions]{Limit laws for rational continued fractions and value distribution of quantum modular forms}
\date{\today}

\author{S. Bettin}
\address{SB: Dipartimento di Matematica, Universit\`a di Genova, via Dodecaneso 35, 16146 Genova, Italy}
\email{bettin@dima.unige.it}

\author{S. Drappeau}
\address{SD: Aix Marseille Universit\'e, CNRS, I2M UMR 7373, 13453 Marseille, France}
\email{sary-aurelien.drappeau@univ-amu.fr}

\begin{document}

\begin{abstract}
  We study the limiting distributions of Birkhoff sums of a large class of cost functions (observables) evaluated along orbits, under the Gauss map, of rational numbers in~$(0,1]$ ordered by denominators. We show convergence to a stable law in a general setting, by proving an estimate with power-saving error term for the associated characteristic function. This extends results of Baladi and Vallée on Gaussian behaviour for costs of moderate growth.
  
  We apply our result to obtain the limiting distribution of values of several key examples of quantum modular forms. We obtain the Gaussian behaviour of central values of the Esterman function~$\sum_{n\geq 1} \tau(n) \e^{2\pi i n x}/\sqrt{n}$ ($x\in \Q$), a problem for which known approaches based on Eisenstein series have been so far ineffective.
  We give a new proof, based on dynamical systems, that central modular symbols associated with a holomorphic cusp form for~$SL(2,\Z)$ have a Gaussian distribution, and give the first proof of an estimate for their probabilities of large deviations.  We also recover a result of Vardi on the convergence of Dedekind sums to a Cauchy law, using dynamical methods.
\end{abstract}

\subjclass[2010]{11A55 (Primary); 37C30, 11F03, 11F67, 60F05 (Secondary)}

\keywords{Gauss map, continued fractions, modular forms, transfer operator, stable laws}

\maketitle

\section{Value distribution of quantum modular forms}

Let~$\Gamma\subset SL(2, \Z)$ be a cofinite Fuchsian subgroup, which acts of functions on~$\P^1(\Q)$ by the weight-$k$ ``slash operator''
$$ f|_k\gamma(x):=(cx+d)^{-k}f(\gamma x) \qquad \text{if } \gamma=\pmat{a & b\\ c& d}\in \Gamma, $$
where~$\gamma x = \frac{ax+b}{cx+d}$ is the Möbius transformation.

In their simplest guise, quantum modular forms, introduced by Zagier~\cite{Zagier2010} (see~\cite{Zagier1999} for early examples), denote the set of functions
$$ f: \P^1(\Q) \setminus S \to \C, $$
for some finite set $S$, satisfying a form of modularity, in the purposely vague sense that for all~$\gamma\in \Gamma$, the function of~$x\in\P^1(\Q) \setminus (S\cup \gamma^{-1}S)$ defined by
\begin{equation}
  h_\gamma(x) := f(x) - f|_k\gamma(x),\label{eq:qmf-mod}
\end{equation}
has some regularity property. Part of the research effort has focused on constructing examples in interrelated ways:
\begin{itemize}
  \item generating series associated with combinatorial sequences: Fishburn matrices~\cite{BringmannLiEtAl2014}, unimodal sequences~\cite{BringmannEtAl2015, KimLimEtAl2016}, partition theory~\cite{NgoRhoades2017} (we note that in the latter, quantum modularity is actually a crucial tool for the asymptotic estimation of partition-related sequences),
  \item radial limits of modular objects (mock theta function, quasi-modular forms) defined on the hyperbolic disk~\cite{Zagier2001, ChoiEtAl2016, Folsom2014, BringmannLovejoyEtAl2018, FolsomEtAl2013},
  \item Eichler integrals, periods of modular functions~\cite{BringmannRolen2016, BringmannKaszianEtAl2019}; state integrals involving the quantum dilogarithm function~\cite{LawrenceZagier1999, DimofteGukovEtAl2009, GaroufalidisKashaev2015},
  \item describing the homology of spaces of cusp forms~\cite{BruggemanEtAl2015,ChoiLim2016,BruggemanChoieEtAl2018},
  \item Kashaev knot invariants and Nahm sums~\cite{Zagier2010,GaroufalidisZagier,BettinDrappeaua},
  \item correlations of the fractional part functions appearing as covariances in the Nyman-Beurling reformulation of the Riemann hypothesis~\cite{Bettin2013a, BettinConrey2013, BalazardMartin2015, LewisZagier2019},
  \item Diophantine approximation and multifractal analysis~\cite{JaffardMartin2018, RivoalRoques2013}.
\end{itemize}

In the present paper we are concerned with the following problem: given a quantum modular form~$f$, how do the multi-sets
\begin{equation}
  \{ f(x), x\in \Q \cap (0, 1], \denom(x)\leq Q \}, \label{eq:set-distrib}
\end{equation}
appropriately normalized, distribute as~$Q\to \infty$?
This topic is tightly related to weak limits of partial sums of certain arithmetic functions, which goes back to Hardy-Littlewood~\cite{Hardy1914}, and has been since then periodically revisited: we mention in particular the works~\cite{Wilton1933, Jurkat1983, Marklof1999} on theta sums, and~\cite{Bettin2015, Maier2016} on cotangent sums. These works are all concerned with instances of QMF of \emph{non-zero} weight.

Our interest in this question comes from the statistical study of additive twists central values of~$L$ functions, which as we will see are \emph{weight-zero} QMF related to the third item described above (periods of modular functions). We are aware of two occurences of this setting in the literature. The first is a result of Vardi~\cite{Vardi1993} on the existence of a limiting distribution for Dedekind sums. The second is a recent result of Petridis and Risager~\cite{PetridisRisager2018} on the distribution of modular symbols, which was motivated by conjectures of Mazur-Rubin~\cite{Mazur2016} and Stein~\cite{Stein2015}. Both results exploit a close connection with Fourier analysis of the modular surface (the spectral analysis of the hyperbolic Laplacian). It is unclear how to extend these methods to more general QMF.

In the present paper we present an approach, based on dynamical systems and the spectral properties of a family of transfer operators, which allows us to answer the question~\eqref{eq:set-distrib} for essentially all level~$1$ (\textit{i.e.} $\Gamma = SL(2, \Z)$) and weight~$0$ QMF.

We denote~$\sigma = \smat{&-1\\1&}$ and $\tau = \smat{1&1 \\ & 1}$ the two usual generators of~$SL(2, \Z)$, so that the associated period functions~\eqref{eq:qmf-mod} with weight~$k=0$ are given by
\begin{align*}
  h_\tau(x) = {}& f(x) - f(x+1), \\
  h_\sigma(x) = {}& f(x) - f(-1/x). \numberthis\label{eq:recip-hsigma}
\end{align*}
Our result may be stated informally as follows.

\begin{theorem}\label{th:distrib-qmf}
  Let~$f$ be a weight~$0$ QMF for~$SL(2, \Z)$, in the sense that~$h_\sigma$ extends to a Hölder-continuous function on~$(0, 1]$ with some regular growth behaviour at~$0$, and~$h_\tau = 0$.
  Then, up to a suitable renormalization, the multi-sets \eqref{eq:set-distrib} become equidistributed according to a stable law, which is characterized by the growth of~$h_\sigma$ at~$0$.
\end{theorem}

The hypothesis are stated more precisely in Theorem~\ref{th:main-general} below. The restriction to~$x\in (0, 1]$ and the assumption~$h_\tau = 0$ are made to clarify the statement, but are inessential (cf. footnote~\ref{fn:general_h}, page~\pageref{fn:general_h}) and are natural in applications. However, the restriction to~$\Gamma = SL(2,\Z)$ is important in our argument. What is required is that the action of~$\Gamma$ on~$\P^1(\R)$ can be induced into an expanding Markov map (the Gauss map in the case~$\Gamma = SL(2, \Z)$). The restriction to weight zero QMF is also natural, as the problem for non-zero weights is of different nature and typically simpler; see~\cite{appA} for more details.

We will find in practice that any bound of the shape~$h_\sigma(x) = O(x^{-1/\alpha})$ as~$x\to 0$ ($\alpha > 2$) ensures convergence to a Gaussian law. A bound of type~$h_\sigma(x) \asymp x^{-1/\alpha + o(1)}$ for some~$\alpha \in (0, 2)$, will typically imply the convergence to a stable law of parameter~$\alpha$.

Due to the relative weakness in its hypotheses, Theorem~\ref{th:distrib-qmf} applies to a wide class of QMF. In the next section, using Theorem~\ref{th:distrib-qmf}, we answer the distribution question for several arithmetic invariants. We expect further applications to follow in the future.

\section{Applications}

For all~$Q\geq 1$, we endow the set
$$ \Omega_Q := \{x = a/q,\ 1\leq a \leq q \leq Q, (a,q) = 1\} \subset\Q\cap(0,1] $$
with the uniform probability measure~$\P_Q$, and we denote~$\E_Q$ the associated expectation,
$$ \E_Q(f(x)) = \abs{\Omega_Q}^{-1} \sum_{x \in \Omega_Q} f(x). $$

\subsection{Central values of the Estermann function}

Let
$$ \tau(n) = \sum_{d\mid n}1 \qquad (n\geq 1) $$
denote the divisor function. The Estermann function, introduced in 1930~\cite{Estermann1930}, is defined by
$$ D(s,x) := \sum_{n\geq1}\frac{\tau(n)\e(nx)}{n^s} $$
for $\Re(s)>1$, and extended by analytic continuation otherwise. It was initially introduced in relation with the shifted divisor problem~$\sum_n \tau(n) \tau(n+1)$. Its functional equation still serves as a basic tool to derive Voronoï summation formulae, which are then used in conjunction with the circle method to study moments of~$L$-function and their arithmetics applications: we mention the proportion of critical zeroes of~$\zeta$~\cite{Conrey1989}, the binary divisor problems~\cite{Motohashi1994}, and non-vanishing of central~$L$-values~\cite{Luo2015}.

We mention a further connection with moments of Dirichlet $L$-functions~$L(s, \chi) = \sum_{n\geq 1} \chi(n) n^{-s}$. By~\cite[Theorem~5]{Bettin2016}, the twisted second moment of Dirichlet~$L$-functions satisfies
\begin{equation}
  \label{eq:link-Mx-Dx}
  \begin{aligned}
    M(a,q) := {}& \frac{1}{q^{1/2}}\sum_{\chi\mod q}|L(\tfrac12,\chi)|^2\chi(a) \\
    = {}& \Re D(\tfrac12,\tfrac aq)+\Im D(\tfrac12,\tfrac aq)+O(q^{-1/2})
  \end{aligned}
\end{equation}
for~$q$ prime and~$q\nmid a$. From this expression, we see that the second moment~$\sum_{a\mod{q}}\abs{M(a, q)}^2$ is essentially the fourth moment of Dirichlet~$L$-functions~$\sum_{\chi\mod{q}} |L(\frac12, \chi)|^4$, whose full evaluation in~\cite{Young2011} lies at the threshold of current techniques of analytic number theory (see~\cite{BlomerFouvryEtAl2017,BlomerEtAl2017} for further work on this topic). This fits in the general problem of understanding the distribution of central values of~$L$-functions and their twists, which is a fundamental topic in analytic number theory~\cite{Selberg1992, ConreyFarmerEtAl2005, Soundararajan2009, Harper2013, RadziwillSoundararajan2015}. Up to now, essentially all known results have been obtained by the moments method.

Using Theorem~\ref{th:distrib-qmf}, we obtain the following Central Limit Theorem.

\begin{theorem}\label{thm:estermann}
  For all~$\eps>0$,~$Q\geq 3$, and all rectangle~$\cR \subset \C$, we have
  \begin{equation*}
    \P_Q\Big(\frac{D(\tfrac12, x)}{\sigma(\log Q)^{\frac12}(\log \log Q)^{\frac32}} \in \cR\Big) = \int_{v_1+iv_2\in\cR}\frac{\e^{-(v_1^2+v_2^2)/2}\df v_1\df v_2}{2\pi} + O_\eps\Big(\frac1{(\log\log Q)^{1-\eps}}\Big) %\label{eq:estim-estermann}
  \end{equation*}
  where~$\sigma = 1/\pi$.
\end{theorem}

We will obtain this result as an application of Theorem~\ref{th:distrib-qmf}, using the fact, proved in~\cite{Bettin2016},  that $D(\frac12,x)$ is a weight-$0$ QMF, the associated function~$h_\sigma$~\eqref{eq:recip-hsigma} being a $(\frac12-\eps)$-H\"older continuous function on~$(0,1]$.

The simplicity with which we will deduce Theorem~\ref{thm:estermann} from Theorem~\ref{th:distrib-qmf} contrasts with the fact that other methods appear to be completely ineffective with this problem.
The moments method (and therefore also the approach of~\cite{Nordentoft2018} described in the following section) can not be applied, due to the presence of a negligible proportion of~$x\in\Omega_Q$ with abnormally large continued fraction coefficients, whose contribution dominates the integer moments of~$D(\frac12, x)$. In fact, all moments of~$D(\frac12, x)$ and~$M(a, q)$ have recently been computed in~\cite{Bettin}: starting already from the second they grow faster that what is suggested by Theorem~\ref{thm:estermann}.

Another tentative approach to Theorem~\ref{thm:estermann} consists in obtaining a limiting distribution result for~$\sum_{n\geq 1}\tau(n)\e(nz)$, where~$z=x+iy$; $x\in[0,1]$ is chosen at random and~$y\to 0^+$, and transfering these properties to its discrete counterpart $D(\frac12, x)\approx \sum_{n\leq q^2} \tau(n)\e(nx)n^{-1/2}$, for~$x=a/q\in\Q$, by Fourier expansion. This method is employed for the incomplete Gauss sum in~\cite{DemirciAkarsu2014}, based on~\cite{Marklof1999}. In our setting, however, connecting the continuous and the discrete averages raises several additional difficulties, among which the problem of dealing with a divergent second moment as well as those coming from considering a central value rather than an object corresponding to an $L$-value off the line. Also the incomplete Gauss sums are not Gaussian distributed, but rather are distributed as the push-forward measure by a theta series. This suggests a difference in nature between the two problems and, notwithstanding the technical difficulties, it suggests this method could not be adapted to our case.

It would be interesting to obtain a statement analogous to Theorem~\ref{thm:estermann} for the values~$M(a,q)$, however the identity \eqref{eq:link-Mx-Dx} as stated holds only for $q$ prime; the corresponding identity for generic~$q$ involves a multiplicative convolution, which is not clearly accountable for using the present method. However we still believe $M(a,q)$ is distributed according to a normal law.
\begin{conjecture}
  As $q\to\infty$ along primes, the multi-set
  $$ \Big\{\frac{M(a,q)}{(\log Q)^{1/2}(\log \log Q)^{3/2}},\ 1\leq a < q \Big\} $$
  become distributed according to a dilated centered Gaussian.
\end{conjecture}

\subsection{Modular symbols}

Out of many possible ways, modular symbols~\cite{Manin2009} can be seen as elements of the space of linear forms on $S_k(\Gamma_0(N))$, the vector space of cusp forms of weight $k\equiv 0\pmod{2}$ and level~$N\geq 1$, spanned by the Shimura integrals
\begin{equation*}
  f \mapsto \syb{x}_{f,m}:=\frac{(2\pi i)^{m}}{(m-1)!}\int_x^{i\infty} f(z)(z-x)^{m-1} \df z=:\syb{x}_{f,m}^++i\syb{x}_{f,m}^-
\end{equation*}
for any $1\leq m \leq k-1$,  $x\in\Q$, $f\in S_k(\Gamma_0(N))$, and where $\syb{x}_{f,m}^\pm\in\R$. (In this section and section~\ref{sec:mod-symb} only, the letter~$f$ will denote a holomorphic cusp form).

Up to an explicit factor, the value $\syb{x}_{f,m}$ is also the special value $L(f,x,m)$ of the analytic continuation of the $L$-function
$$ L(f,x,s) := \sum_{n=1}^\infty a_n \e(nx)n^{-s} \qquad (\Re s > k) $$
where we write~$f(z) =\sum_{n\geq1}a_n \e(nz)$ for~$\Im(z)>0$. Being at the intersection of the geometric, modular and arithmetic aspects of~$\Gamma_0(N)$, modular symbols received a considerable amount of interest, \emph{e.g.} for computing with modular forms~\cite{Cremona1997, Stein2007}. The central values of the Estermann function~$D(\frac12, x)$, which was the subject of the previous section, may be interpreted as a regularized modular symbol for the derivative $\frac{\partial}{\partial s}E_2(z,s)\vert_{s=1/2}$ of the Eisenstein series~\cite[chapter~3.5]{Iwaniec2002}.

The symbol associated with the central value~$m = k/2$ plays a particular rôle and is the focus of this section. Motivated by the $abc$-conjecture, the question of the size of modular symbols was initially asked by Goldfeld~\cite{Goldfeld1999} and later studied by Petridis and Risager~\cite{Petridis2002,PetridisRisager2004}. Interest in this question was recently revived by questions of Mazur, Rubin and Stein~\cite{Mazur2016,Stein2015}, motivated in part by the growth of ranks of elliptic curves. In particular, for $f$ a fixed form of weight $k=2$ (and so $m=1$) and $x\in (0, 1]$ varying along rationals of reduced denominator $q$, with $q\to\infty$, Mazur and Rubin predicted that
$$ \Big\{\frac{\syb{a/q}^\pm_{f,1}}{\sqrt{\log q}}, 1\leq a \leq q, (a,q)=1\Big\} $$
becomes asymptotically distributed according to a suitably dilated normal law. To our knowledge, only the first and second moments have been computed~\cite[Chapter~9]{BlomerEtAl2018}.

The situation changes with an additional average over~$q\leq Q$. Then the Central Limit Theorem for weight~$k=2$, level~$N\geq 1$ forms has been proved by Petridis and Risager~\cite{PetridisRisager2018}, by using the spectral analysis of twisted Eisenstein series and the location of eigenvalues of the hyperbolic Laplacian.

Using Theorem~\ref{th:distrib-qmf}, we will prove the following Central Limit Theorem for modular symbols associated with forms of level~$N=1$ and arbitrary weight~$k\geq 12$.

\begin{theorem}\label{thm:modsymbol}
  Let~$k\geq 12$ be even, and~$f\in S_k(SL(2,\Z))\smallsetminus\{0\}$ be fixed. Then for all~$Q\geq 2$ and rectangle~$\cR\subset \C$, we have
  \begin{equation}
    \P_Q\Big(\frac{\syb{x}_{f,\frac k2}}{\sigma_f\sqrt{\log Q}} \in \cR \Big) = \int_{v_1+iv_2\in\cR}\frac{\e^{-(v_1^2+v_2^2)/2}\df v_1\df v_2}{2\pi} + O\Big(\frac1{\sqrt{\log Q}}\Big), \label{eq:estim-modsymbol}
  \end{equation}
  with~$\sigma_f^2 = \frac{3 (4\pi)^k}{\pi \Gamma(k)} \norm{f}_k^2$, where~$\norm{f}_k$ is the weight-$k$ Petersson norm of~$f$.
\end{theorem}

For example, this result applies for~$k=12$ with $f$ being the discriminant modular form $\Delta(z)$. The error term is optimal and uniform with respect to~$\cR$.

Using Theorem~\ref{th:distrib-qmf}, we will also deduce the following statement, which can be interpreted as a Large Deviations Principle.

\begin{theorem}\label{thm:modsymbol-LD} Let~$f$ be fixed as in Theorem~\ref{thm:modsymbol}. For any~$\eps>0$, we have
  \begin{equation}
    \limsup_{Q\to \infty} \P_Q\Big(\abs{\syb{x}_{f, \frac k2}} > \eps\log Q\Big)^{1/\log Q} < 1.\label{eq:modsym-ldp}
  \end{equation}
\end{theorem}

We will deduce Theorems~\ref{thm:modsymbol} and \ref{thm:modsymbol-LD} from Theorem~\ref{th:distrib-qmf} by showing that the map~$x\mapsto \syb{x}_{f,\frac k2}$ is a weight~$0$ QMF with the period function~$h_\sigma$~\eqref{eq:recip-hsigma} being Hölder-continuous on~$[0,1]$. We will estimate all small exponential moments, from which the Large Deviations Principle will follow. 

Theorem~\ref{thm:modsymbol} is related to three works, two of which have been carried out around the same time as the present.

\begin{itemize}
  \item The work of Petridis and Risager~\cite{PetridisRisager2018} was concerned with the case of weight~$k=2$ forms, which is relevant to the conjectures of Mazur-Rubin and Stein (cusp forms associated with elliptic curves). Their method is based on the analysis of twisted Eisenstein series. Although the authors do not seem to mention it, their method is capable of obtaining a Large Deviations Principle. The assumption $k=2$ is however crucial in their approach, since only in this case the analogues of the function~$h_\sigma$ in~\eqref{eq:recip-hsigma} are constant (this translates into the independence of the multiplier system defined in~\cite[p.~7]{PetridisRisager2018} with respect to~$z_0$).

  \item Very recently, Nordentoft~\cite{Nordentoft2018} has obtained the Central Limit Theorem for arbitrary weight~$k\geq 2$ and level~$N\geq 1$ modular forms. Specializing to~$N=1$, this gives in particular an independent proof of Theorem~\ref{thm:modsymbol}. Similarly as~\cite{PetridisRisager2018}, the spectral analysis of twisted Eisenstein series play a central rôle, however his argument is based on a completely different construction (due to the ``independence with respect to~$z_0$'' obstruction mentioned above). As a result, it crucially relies on the consideration of moments, and falls short of establishing a Large Deviations Principle (Theorem~\ref{thm:modsymbol-LD}).

  \item Also very recently, Lee and Sun~\cite{LeeSun2018} have independently obtained a proof of the main theorem of~\cite{PetridisRisager2018}, for weight~$k=2$, by a method similar to the one we pursue here. This is achieved by considering a certain twisted version of the Gauss map, keeping track at each iteration of a coset in~$\Gamma_0(N)\backslash SL_2(\Z)$. On the other hand, for~$k=2$ the period functions (the analogues of~$h_\sigma$ in~\eqref{eq:recip-hsigma}) are constant. The technical difficulties they are faced with are of very different nature from the ones we encounter here. By mixing the methods presented here with those of~\cite{LeeSun2018}, it is plausible that the main results of~\cite{Nordentoft2018}, for arbitrary weight and levels, could be recovered by dynamical methods.
\end{itemize}

The difference in the magnitude of the variance between Theorems~\ref{thm:estermann} and \ref{thm:modsymbol-LD} is noteworthy, and is due to the presence, or not, of a pole the associated $L$-functions ($\zeta^2(s)$ and $L(f, s)$ respectively) at~$s=1$.

We have considered only the modular symbols~$\syb{x}_{f, m}$ at the central value~$m=k/2$. The behaviour for~$m\neq k/2$ is very different and simpler: for instance, when~$m>k/2$, the series~$L(f, x, m)$ extends to a continuous function of~$x\in \R$, and the values~$\{\syb{x}_{f, m}, x\in \Omega_Q\}$ become distributed according to the push-forward~$(L(f,\cdot,m))_*(\df\nu)$ of the Lebesgue measure. We return to this question in more details in~\cite{appA}.

The original conjectures of Mazur-Rubin and Stein were concerned with the case of a single~$q$. There does not seem to be an effective way, with our method or those of~\cite{PetridisRisager2018,Nordentoft2018,LeeSun2018}, to isolate rationals of fixed denominators, whence all results proven so far rely on the extra average over~$q$.

\subsection{Kashaev invariants of the~$4_1$ knot and sums of continued fraction coefficients}

Our next application is motivated by a question in~\cite{Zagier2010}.
To a knot~$K$ and an integer $n \geq 2$, we associate the $n$-colored Jones polynomial~\cite[section~1]{Garoufalidis2018}, which is a Laurent polynomial $J_{K, n}(\mathbf{q}) \in\Z[[\mathbf{q}]]$ defined by a combinatorial construction through a diagram representation of~$K$. For any root of unity $\mathbf{q}\in\C$, we define~$J_{K,0}(\mathbf{q}) := J_{K,n}(\mathbf{q})$ where~$n\geq 1$ is the order of~$\mathbf{q}$. In~\cite{MurakamiMurakami2001}, it was shown that the sequence~$(J_{K, 0}(\e^{2\pi i/n}))_{n\geq 1}$ is the Kashaev invariant of~$K$~\cite{Kashaev1995} (the full function can be reconstructed by the action of the Galois group).
%In~\cite{MurakamiMurakami2001}, this was shown to be equal to Galois orbits of the Kashaev invariants.
In the case of the $4_1$ knot (or ``figure-eight'' knot), the simplest hyperbolic knot, we have explicitely~\cite{Zagier2010}
\begin{equation*}
  \J(x) := J_{4_1,0}(\e(x)) = \sum_{m=0}^{\infty}|1-\e^{2\pi i x}|^2\cdots|1-\e^{2\pi i m x}|^2,\qquad (x\in\Q).
\end{equation*}
Note that for each given~$x\in\Q$, the sum is finite. In this case, Zagier's modularity conjecture, stated precisely in~\cite{Zagier2010}, predicts that~$x\mapsto \log \J(x)$ is a weight~$0$ QMF: the difference
\begin{equation*}
  h(x) := \log \J(-1/x) - \log \J(x),%\label{eq:def-periodfun-J}
\end{equation*}
which is depicted in~\cite[Fig.~4]{Zagier2010} is expected to behave ``nicely'' with respect to~$x$, although not continuously. Conjecturally, we expect~$h(x) \sim C/x$ as~$x\to 0$, where~$C ={\rm Vol}(4_1)/2\pi$ and~${\rm Vol}(4_1) = 2.02\dotsb$ is the hyperbolic volume of~$4_1$; Kashaev's volume conjecture is the case~$x = 1/n$, $n\in \N$, which is known in this case~\cite{AndersenHansen2006}.

A proof of Zagier's conjecture for the~$4_1$ knot has been announced by Garoufalidis and Zagier~\cite{GaroufalidisZagier}. In~\cite{BettinDrappeaua}, we obtained independently another proof, complemented by a reciprocity formula relative to a transformation of another kind,
%(essentially the conjugate of the Gauss map by~$a/q \mapsto \bar{a}/q$, where~$a\bar{a}\equiv 1\mod{q}$).
by which we deduced the following asymptotic estimate: for $\lambda\geq 0$, denote
\begin{equation}
  \Sigma_\lambda(x) = \sum_{j=1}^{r(x)} a_j^\lambda, \qquad x \in \Q\cap (0, 1],\quad x = [0; a_1, \dotsc, a_{r(x)}],\ a_{r(x)}>1,\label{eq:def-Sigmalambda}
\end{equation}
the sum of continued fraction coefficients of~$x$. Then by~\cite{BettinDrappeaua}, we have
\begin{equation}
  \log \J(x) \sim C \Sigma_1(x)\qquad \text{as} \qquad \frac{\Sigma_1(x)}{r(x)} \to \infty.\label{eq:rel-J-sigma}
\end{equation}
This is in accordance with the conjectured behaviour of~$h(x)$ as~$x\to 0$.

The map~$x\mapsto \Sigma_\lambda(x)$, suitably extended to~$\Q$, is a weight~$0$ QMF with associated period function~\eqref{eq:recip-hsigma} satisfying~$h_\sigma(x) = C\floor{1/x}^\lambda$ for~$x\in \Q\cap (0, 1]$. Using Theorem~\ref{th:distrib-qmf}, we compute in Theorem~\ref{th:m_lambda} below the distribution of $\Sigma_\lambda$ for all $\lambda\geq0$, extending work of~\cite{BaladiVallee2005,Baladi.Hachemi2008}.
In particular, in the case $\lambda=1$ we obtain the following result.

\begin{theorem}\label{th:moment-1}
  Let~$G_1$ denote the cumulative distribution function of the stable law~$S_1(\frac6\pi, 1, 0)$, that is~$G_1(v) := \int_{-\infty}^v g_1(x) \df x$ where
  $$ g_1(x) = \frac1{2\pi} \int_{-\infty}^\infty \e^{-itx}\e^{-\frac6\pi \abs{t} - \frac{12}{\pi^2} i t \log \abs{t}} \df t, $$
  and let~$\gamma_0$ be the Euler--Mascheroni constant. Then, as~$Q\to \infty$,
  \begin{equation}
    \P_Q\Big(\frac{\Sigma_1(x)}{\log Q} - \frac{\log\log Q - \gamma_0}{\pi^2/12} \leq v\Big) = G_1(v) + O\Big(\frac1{(\log Q)^{1-\eps}}\Big).\label{eq:discrete-heinrich}
  \end{equation}
\end{theorem}

Theorem~\ref{th:moment-1} answers a question in~\cite{FlajoletEtAl}, and echoes a result of Heinrich~\cite{Heinrich1987} which obtained a similar convergence for Lebesgue-almost all~$x\in[0, 1]$. In particular,~\eqref{eq:discrete-heinrich} implies that the generic time complexity of the substractive algorithm for the GCD is asymptotically~$(1+o(1))\frac{12}{\pi^2}(\log Q) \log \log Q$ on pairs of coprime numbers at most~$Q$, as~$Q\to\infty$. This is in sharp contrast with the average time complexity, which is~$(1+o(1))\frac{6}{\pi^2}(\log Q)^2$ (the latter is known even with a single average over numerators, see~\cite{YaoKnuth1975}). Note that, contrary to the case of the Estermann function (Theorem~\ref{thm:estermann}), this discrepancy between the typical size and the average size is consistent with fact that the stable law~$S_1(\frac6\pi, 1, 0)$ has a divergent first moment.

Combined with~\eqref{eq:rel-J-sigma}, Theorem~\ref{th:moment-1} implies the following law of large numbers for the values of~$\log \J(x)$.

\begin{corollary}\label{cor:LLN-J}
  As~$Q\to \infty$, we have
  \begin{equation*}
    \log \J(x) \sim \tfrac{12}{\pi^2} C \log Q \log\log Q%\label{eq:LLN-J}
  \end{equation*}
  for a proportion~$1-o(1)$ of fraction~$x\in\Q\cap(0, 1]$ of denominators at most~$Q$.
\end{corollary}

The typical size of~$\log \J(x)$ among fractions~$x$ with~$\denom(x)\leq Q$ is therefore much less than that of~$\log \J(1/Q) \sim C Q$ (by~\eqref{eq:rel-J-sigma}).

We expect a convergence in law analogous to~\eqref{eq:discrete-heinrich} for the values~$\log\J(x)$, however the error term which we obtain in~\eqref{eq:rel-J-sigma} is not precise enough to carry this out. This issue is discussed more precisely in~\cite{BettinDrappeaua}, along with the case of other knots.

\subsection{Dedekind sums}

For all coprime integers~$a$ and~$q$ with~$q\geq 1$, the Dedekind sum
$$ s\Big(\frac aq\Big) := \sum_{h=1}^{q-1} \syf{\frac{ha}q} \syf{\frac hq}, \qquad \syf{x} := \begin{cases}  \{ x \} - 1/2 & (x \not\in \Z), \\ 0 & (\text{otherwise}) \end{cases} $$
is a rational number of modulus at most~$q/12+1/2$. They appear naturally in the multiplier system attached to the Dedekind~$\eta$ function; we refer to the monograph~\cite{Rademacher1972} for further properties and references.
The value distribution of~$s(x)$ on average over rational~$x$ has been studied from several points of view~\cite{Hickerson1977,Vardi1987,Bruggeman1990,Vardi1993}.

From Theorem~\ref{th:distrib-qmf}, we will deduce a short proof of the following convergence to a Cauchy law, which is the main result of~\cite{Vardi1993}.
\begin{theorem}[\cite{Vardi1993}]\label{th:dedekind}
  Uniformly for~$v\in\R$ and~$Q\geq 2$, we have
  \begin{equation}
    \P_Q\Big(\frac{s(x)}{\log Q} \leq \frac{v}{2\pi} \Big) = \frac1\pi \int_{-\infty}^{v} \frac{\df y}{1+y^2} + O\Big(\frac1{(\log Q)^{1-\eps}}\Big).\label{eq:estim-dedekind}
  \end{equation}
\end{theorem}

This statement will easily follow from Theorem~\ref{th:distrib-qmf} by noting that~$x \mapsto s(x)$ is a weight~$0$ QMF, with associated period function~$h_\sigma$~\eqref{eq:recip-hsigma} being roughly $\floor{1/x} $.\footnote{Notice that we obtain a different distribution than in Theorem~\ref{th:moment-1}. This is due to the fact in that case the period relation is more precisely $f(x)+f(-1/x)=\floor{1/x} $ rather than $f(x)-f(-1/x)=\floor{1/x}$.}
This last fact is a consequence of the reciprocity formula for~$s(x)$~\cite[Chapter~2, Theorem~1]{Rademacher1972}.

The proof of~\cite{Vardi1993} builds on the close connection between~$s(x)$ and the multiplier system associated to the~$\eta$ function, which brings the problem to an analysis of twisted Poincaré series on the modular surface, which are in turn studied by means of the Kuznetsov trace formula. By contrast, our arguments are dynamical in nature, and use little arithmetic information beyond the group structure of~$SL(2,\Z)$.

\section{Overview}

\subsection{Reduction to dynamical analysis}

We now overview the arguments underlying Theorem~\ref{th:distrib-qmf}. Suppose~$f: \Q \to \C$ is such that the functions defined through~\eqref{eq:qmf-mod} satisfy~$h_\tau = 0$ and~$h_\sigma$ extends to a continuous function on~$(0, 1]$.
To simplify the presentation, in this section only, we assume that~$f$ is an even function.

The starting point is the remark that the action of~$\sigma$ and~$\tau$ on~$\P^1(\Q)$ can be induced into an expanding Markov map on~$(0, 1]$, namely the Gauss map~$T:x \mapsto \{1/x\}$, where the braces denote the fractional part. More precisely, by periodicity and since~$f$ is even, we have~$f(x) = h_\sigma(x) + f(\{1/x\})$, which after iteration (Euclid's algorithm) yields\footnote{Without the assumption $h_\tau=0$ a similar formula holds: we have in general
$f(x) = \sum_{j=1}^{r(x)} H(T^{j-1}(x)) + f(0),$
where $H(x):=h_\sigma(x)-\sum_{i=1}^{\floor{1/x}}h_\tau(1/x-i)$ for $x\in(0, 1]$, and the ensuing analysis still applies.\label{fn:general_h}}

\begin{equation}
f(x) = \sum_{j=1}^{r(x)} h_\sigma(T^{j-1}(x)) + f(0) \qquad (x\in (0, 1])\label{eq:f-sumh}
\end{equation}
where~$r(x) \geq 0$ is minimal subject to~$T^{r(x)}(x) = 0$.

This expression relates the values of~$f$ to Birkhoff sums associated with the Gauss map~$T$. This map is known to have good mixing properties~\cite[p.~174]{CornfeldEtAl1982}, by which we expect the sum~\eqref{eq:f-sumh} to behave like a sum of independent random variables.
Precise limiting behaviour may be obtained through the study of spectral properties of transfer operators associated with~$T$: this is an important theme in smooth dynamics and stationary Markov chains~\cite{Doeblin1940,Fortet1940,IonescuTulceaMarinescu1950,Wirsing1973/74,Rousseau-Egele1983,Broise1996,AaronsonDenker2001}. We refer in particular to~\cite{Broise1996} and to the introduction of~\cite{AaronsonDenker2001} for an extensive historical account and references.
Among maps of the interval, the Gauss map has been particularly studied because of its link with the analysis~\cite{Mayer1976,Pollicott1986,Mayer1991a,MomeniVenkov2012} of geodesic flows on the surface~$SL(2,\Z)\backslash {\mathfrak h}$ (where~${\mathfrak h}$ is the upper-half plane).

A prominent example is given by~\cite[Theorem~8.1]{Broise1996}: suppose~$\phi:[0,1]\to\R$ is of bounded variations and not of the shape~$c + k - k\circ T$, for some function~$k$ of bounded variations and some constant~$c\in\R$, and let
$$ S_N(\phi, x) := \sum_{j=1}^N \phi(T^j(x)). $$
Then for some constants~$\mu_\phi\in\R$ and~$\sigma_\phi\in\R_+^*$, uniformly for~$v\in\R$,
\begin{equation*}
  \P\Big(\frac{S_N(\phi, x) - \mu_\phi N}{\sigma_\phi \sqrt{N}} \leq v\Big) = \Phi(v) + O\Big(\frac1{\sqrt{N}}\Big), \qquad \Phi(v) := \int_{-\infty}^v \frac{\e^{-t^2}\df t}{\sqrt{2\pi}} 
\end{equation*}
as $N\to\infty$, where~$x\in(0,1)$ is chosen uniformly according to the Lebesgue measure. The implied constant may depend only on~$\phi$. This relies on the spectral analysis of perturbations of the Gauss-Kuzmin-Wirsing transfer operator
$$ \H[f](x) = \sum_{n\geq 1} \frac1{(n+x)^2}f\Big(\frac1{n+x}\Big). $$

We are interested in the case when~$x$ is a rational chosen at random in~$\Omega_Q$, and the Birkhoff sum in~\eqref{eq:f-sumh} is over the full orbit, its length varying with~$x$. Denote
\begin{equation}
  S_\phi(x) := \sum_{j=1}^{r(x)} \phi(T^{j-1}(x)).\label{eq:def-Sphi}
\end{equation}
In~\cite{Vallee2000}, Vallée has shown that the expectations of~$S_\phi(x)$ satisfy
$$ \E_Q(S_\phi(x)) = \mu_\phi \log Q + \nu_\phi + O(Q^{-\delta}) $$
for functions~$\phi$ which are constant on each interval~$(\frac1{n+1}, \frac1n)$ ($n\geq 1$) and under some growth condition at~$0$. Here the number~$\delta>0$ is absolute and the implied constant may depend on~$\phi$. The numbers~$\mu_\phi, \nu_\phi$ depend only on~$\phi$, and in fact
$$ \mu_\phi = \frac{12\log 2}{\pi^2}\int_0^1 \phi(x)\xi(x)\df x, \qquad \xi(x) := \frac1{(1+x)\log 2}. $$
The quantity~$\mu_\phi$, and the measure~$\xi(x)\df x$, are essentially the projection of~$\phi$ on the eigenspace of~$\H$ associated to its dominant eigenvalue~$1$.
An important point is that this question was studied within the framework of dynamical methods. The argument uses the construction of a suitable generating series involving the quasi-inverse~$(\Id-\H_\tau)^{-1}$ of twisted transfer operators 
\begin{equation}
  \H_{\tau}[f] = \sum_{n\geq 1} \frac1{(n+x)^{2+i\tau}}f\Big(\frac1{n+x}\Big),\label{eq:GKW-twist}
\end{equation}
for arbitrary large~$\tau\in\R$ (by contrast with the continuous setting when~$x$ was chosen uniformly at random in~$[0, 1]$, which involves perturbations of a single fixed operator). This construction crucially relies on the fact that the denominator~$q(x)$ of~$x$ can be detected by means of the Birkhoff sum~$\log q(x) = -\sum_{j=1}^{r(x)} \log(T^{j-1}(x))$. Earlier approaches~\cite{Heilbronn1969,Dixon1970}, restricted to~$\phi=1$, involved number-theoretic methods based on bounds on algebraic exponential sums.

The approach of~\cite{Vallee2000} was developed by Baladi and Vall\'{e}e~\cite{BaladiVallee2005}, who proved that for~$\phi$ constant on each interval of the shape~$(\frac1{n+1}, \frac1n)$ and under a logarithmic growth condition~$\phi(x) = O(\log(2/x))$, the Laplace transform satisfies the ``quasi-powers expansion''
\begin{equation}
  \E_Q(\e^{w S_\phi(x)}) = \exp\big\{ U(w) \log Q + V(w) + O(Q^{-\delta}) \big\}\label{eq:BV-laplace}
\end{equation}
for~$w$ in some complex neighborhood of~$0$; the holomorphic functions~$U,V$, the number~$\delta>0$ and the implied constant may depend on~$\phi$. By Hwang's theorem~\cite{Hwang1996}, this has a number of consequences in terms of the limiting distribution, and among them, an effective central limit theorem: if~$\phi$ is real, non identically zero and as above, then for some numbers~$\mu_\phi\in\R$, $\sigma_\phi>0$ and all~$v\in\R$ we have
$$  \P_Q\Big(\frac{S_\phi(x) - \mu_\phi \log Q}{ \sigma_\phi (\log Q)^{1/2} } \leq v \Big) = \int_{-\infty}^v \frac{\e^{-t^2/2}\df t}{\sqrt{2\pi}} + O\Big(\frac1{\sqrt{\log Q}}\Big) \qquad (t\in\R). $$
%An explicit formula for the mean value was given above; as far as we know, there is no closed form expression for~$\sigma_\phi$ in general, although it can be approximated in polynomial time (see~\cite{CiupercaGirardinEtAl2011}).
The power-saving error term in~\eqref{eq:BV-laplace} depends on proving a pole-free strip for the quasi-inverse of the twists~\eqref{eq:GKW-twist}, which can be viewed as a ``quasi-Riemann hypothesis'' for the generating function of interest\footnote{The same set of idea imply a zero-free strip for the Selberg zeta function of the full modular group~\cite{Naud2005}.}. The approach of~\cite{BaladiVallee2005} extends seminal work of Dolgopyat~\cite{Dolgopyat1998} to the case of an expanding interval map with an infinite partition (see also~\cite{BaladiVallee2005a}).

Motivated by our application~\eqref{eq:f-sumh}, we extend the methods of~\cite{BaladiVallee2005} in two directions. In the first direction, we consider cost functions~$\phi:(0,1]\to\R$ which are not necessarily constant by parts. We will require that~$\phi$ can be extended to a Hölder continuous function, with some uniform exponent, on each interval~$[\tfrac1{n+1},\tfrac1n]$.

The second direction we wish to consider is cost functions~$\phi(x)$ having a possibly divergent first or second moment, say~$\phi(x) = x^{-1/2}$, or~$\phi(x) = \floor{1/x}$, with the consequence that the limit law will be stable, but not necessarily Gaussian anymore~\cite[p.~384]{BaladiVallee2005}.  This is a well-known theme in the theory of sums of independent random variables; see Chapter~VI of~\cite{Levy1925}, or Chapters~VI.1 and~XVII.5-6 of~\cite{Feller1971}. The corresponding phenomenon for sums of continued fractions coefficients in the continuous setting (Lebesgue-almost all~$x\in (0, 1)$) has been investigated by elementary means by Lévy~\cite{Levy1952} (see also~\cite{Heinrich1987}), and later by transfer operator methods~\cite{GuivarchHardy1988,GuivarchLeJan1993,Szewczak2009,AaronsonDenker2001}. In fact, this falls into the general ``countable Markov-Gibbs'' framework of~\cite{AaronsonDenker2001}, where it is refered to as the ``distributional limit'' problem. A large part of later work has focused on non-uniformly hyperbolic maps; we refer to the survey~\cite{Goueezel2015} and the references therein.

We investigate the corresponding question in the discrete setting ($x$ at random in~$\Omega_Q$). We evaluate with effective, power-saving error terms the characteristic function
$$ \E_Q(\e^{it S_\phi(x)}) $$
for~$t$ in a real neighborhood of~$0$, under hypotheses which essentially reduce to the boundedness of some positive absolute moment,~$\int_0^1 \abs{\phi(x)}^{\alpha_0}\df x<\infty$ for some $\alpha_0>0$.
%Our result involves the integral
%$$ \II_\phi(t) := \int_0^1 (\e^{it\phi(x)}-1)\xi(x)\df x \qquad (t\in\R). $$
% résultats

\subsection{Statement of the main distributional result}

For technical reasons, we will work in a more general setting. For a parameter~$\kappa\in[0, 1]$, a real interval~$\I$ and a normed space~$X$ let~$\Hol^\kappa(\I, X)$ denotes the set of functions~$\I\to X$ such that the Hölder semi-norm
\begin{equation}
  \HN{\kappa}{f} := \sup_{x, y\in \I, x\neq y} \frac{\norm{f(x)-f(y)}}{\abs{x-y}^\kappa}\label{eq:def-holdernorm}
\end{equation}
is finite.

We let~$\cH := \{ h:[0,1]\to\R \mid \exists n\geq1 \text{ s.t. }  h(x) =\frac1{n+x}\}$ be the inverse branches of the Gauss map, and~$\cH^\ell := \{h_1 \circ \dotsb \circ h_\ell \mid h_j \in\cH\}$.

We fix an integer~$m\geq 1$, and~$m$ functions~$\vphi_1, \dotsc, \vphi_m :[0,1]\to\R^d$. We extend this definition by periodicity, letting~$\vphi_j := \vphi_{j\mod{m}}$ for all~$j\geq 1$, and we define a function on~$\Q\cap(0,1]$ by letting~$S(1) := 0$ and
$$ S(x) := \sum_{j=1}^{r(x)} \vphi_j(T^{j-1}(x)), \qquad (x\in\Q\cap(0,1)) $$
where we recall that
$$ r(x) := \min\{j\geq 0, T^j(x) = 0\}. $$
For~$x\in(0,1)$ with~$r(x)\geq m$, we define
\begin{equation}
  \vphi(x) := \sum_{j=1}^{m} \vphi_j(T^{j-1}(x)),\label{eq:def-vgamtot}
\end{equation}

We make the following hypotheses:
\begin{enumerate}
  \item ($\kappa_0$-Hölder continuity) For each~$n\geq 1$ and~$1\leq j \leq m$, the function~$\vphi_j$ can be extended as a~$\kappa_0$-Hölder continuous function on the interval~$[\frac1{n+1},\frac1n]$.
  \item (Norm~$\alpha_0$-th moment) For each~$1\leq j\leq m$, we have
  \begin{equation}
    \sum_{h\in\cH} \abs{h'(0)} \norm{\vphi_j|_{h(\I)}}_{\infty}^{\alpha_0} < \infty.\label{eq:bound-norm-rough}
  \end{equation}
  \item (Hölder~$\lambda_0$-th moment) For each~$1\leq j\leq m$, we have
  \begin{equation}
    \sum_{h\in\cH} \abs{h'(0)} \HN{{\kappa_0}}{\vphi_j|_{h(\I)}}^{\lambda_0} < \infty.\label{eq:bound-Ckappa-rough}
  \end{equation}
\end{enumerate}

For all~${\bm t}\in\R^d$, we denote~$\nvt:=\nvt_\infty$. Finally, let
\begin{equation*}  
  \II_\vphi(\vt) := \int_0^1 (\e^{i\langle\vt, \vphi(x)\rangle}-1)\xi(x)\df x.
\end{equation*}

\begin{theorem}\label{th:main-general}
  Let~$\kappa_0, \alpha_0, \lambda_0>0$ be given with~$\kappa_0, \lambda_0 \leq 1$, and~$\vphi:(0,1]\to\R^d$ satisfying the conditions~\eqref{eq:bound-norm-rough}--\eqref{eq:bound-Ckappa-rough}. There exists~$t_0>0$, $\delta>0$, and functions~$U, V: \{\vt\in\R^d, \nvt \leq t_0\}\to\C$ such that for all~$\vt\in\R^d$ with~$\nvt\leq t_0$, we have
  \begin{equation}
    \E_Q\big(\e^{i\langle\vt, S(x)\rangle}\big) = \exp\big\{U(\vt)\log Q + V(\vt) + O(Q^{-\delta})\big\},\label{eq:asympt-expect-UV}
  \end{equation}
  and
  \begin{align*}
    U(\vt) = {}& \frac{12\log 2}{m\pi^2}\II_\vphi(\vt) + O(\nvt^2+\nvt^{2\alpha_0-\eps}), \numberthis\label{eq:def-psi-1} \\
    V(\vt) = {}& O(\nvt+\nvt^{\alpha_0-\eps}).
  \end{align*}
  If moreover~$\alpha_0>1$, then there exists a real~$d\times d$ matrix~$C_\vphi$ such that
  \begin{equation}
    U(\vt) = \frac{12\log 2}{m\pi^2}\II_\vphi(\vt) + \vt^T C_\vphi \vt + O(\vt^3+\abs{\vt}^{1+\alpha_0-\eps})\label{eq:thmgen-secondterm}
  \end{equation}
  with~$\vt$ interpreted as a column vector and~$\vt^T$ its transpose. The numbers~$\delta, t_0$ and the implied constant depend at most on~$\alpha_0, \kappa_0, \lambda_0, \eps$ and on an upper-bound for the left-hand sides of~\eqref{eq:bound-norm-rough}--\eqref{eq:bound-Ckappa-rough}.
\end{theorem}

\begin{remark}
  \begin{itemize}
    \item It is important to note that the actual values of~$\kappa_0$ and~$\lambda_0$ only affect the statement up to the value of~$t_0, \delta$ and the implied constants. In particular, if~$\vphi$ is~$\CC^1$ on each interval~$[\frac1{n+1},\frac1n]$ and if there exists~$C\geq 1$ such that~$\|\df_x \vphi\| = O(x^{-C})$ for~$x\in(0, 1]$, then~\eqref{eq:bound-Ckappa-rough} is satisfied with~$\kappa_0 = 1$ and any~$\lambda_0<1/C$.
    \item In our applications, we will have~$\|\vphi(x)\| \asymp \|\vphi(y)\|$ for all~$x, y$ in~$h(\I)$, uniformly in~$h\in\cH^m$. In this situation,~\eqref{eq:bound-norm-rough} is equivalent to~$\int_0^1 \|\vphi(x)\|^{\alpha_0}\df x < \infty$.
    \item In typical applications, the quantity~$\II_\vphi(\vt)$ can be evaluated by standard methods, and the results relevant for the present paper are worked out in~\cite{appB}.
    \item The results of~\cite{BaladiVallee2005} are stated in a formalism which includes the Gauss map as a special case. The most crucial assumption, at least as far as one is interested in power-saving error terms, is the ``uniform non-integrability'' assumption~\cite[p.357]{BaladiVallee2005} (see~\cite{Morris2015} for a qualitative result, not using the UNI condition). In the present work, we do not use any more specific properties of the Gauss map; with suitable modifications, the arguments presented here apply to the centered and odd Euclidean algorithms~\cite[Fig. 1]{BaladiVallee2005} as well.
    \item A generalization in a different direction of Baladi-Vallée's results, for maps associated to a reduction algorithm in congruence subgroups, has very recently and independently been obtained by Lee and Sun~\cite{LeeSun2018}\footnote{Another way to study levels~$N>1$ would be by building an expanding map out of Atkin-Lehner homographies, but our attempts to construct such a map satisfying the UNI condition were not successful.}.
  \end{itemize}
\end{remark}

The main feature of Theorem~\ref{th:main-general} we use is the relative weakness of the hypotheses on~$\vphi$. This is important in our applications, where often little is known on~$\vphi$ besides the regularity properties, and rough bounds on the Hölder norms. To obtain this uniformity, we systematically use Hölder spaces, not only because of the regularity of~$\vphi$, but also in order to dampen the oscillations of~$\e^{i\langle \vt, \vphi\rangle}$ (see~\eqref{eq:Ckappa-exp-if} below). This is the main reason why arbitrarily small values of~$\lambda_0$ are admissible for Theorem~\ref{th:main-general}.

The shape of the asymptotic estimates~\eqref{eq:asympt-expect-UV}--\eqref{eq:thmgen-secondterm} are characteristic of infinitely divisible distribution~\cite[Chapter~XVII.2]{Feller1971}. The informally stated Theorem~\ref{th:distrib-qmf} follows, in all the cases we consider, by an evaluation of~$\II_\vphi(\vt)$ and the Berry-Esseen inequality~\cite[Theorem~II.7.16]{Tenenbaum2015a}, \cite[Theorem~IX.5]{FlajoletSedgewick2009}.

The variety of situations we consider prevents us from making a clear and concise set of hypotheses on~$h_\sigma$ which would make Theorem~\ref{th:distrib-qmf} rigorous. Besides working out in details the applications, we restrict to illustrating the case~$m=1$ and~$\alpha_0>2$ by the following Central Limit Theorem for the Birkhoff sums~$S_\phi$ defined in~\eqref{eq:def-Sphi}, which recovers and extends in particular~\cite[Theorem~3.(a)]{BaladiVallee2005} (see also~\cite[Remark~1.3]{Baladi.Hachemi2008}). We recall that~$\Phi$ is the cumulative distribution function of the standard normal law.

\begin{corollary}\label{cor:CLT}
  Assume that the bound
  $$ \sum_{n\geq 1}\frac1{n^2} \Big( \sup_{x\in[\frac1{n+1},\frac1n]}\abs{\phi(x)}^{\alpha_0} + \sup_{x, y\in [\frac1{n+1},\frac1n]} \frac{\abs{\phi(x)-\phi(y)}^{\lambda_0}}{\abs{x-y}^{\kappa_0\lambda_0}}\Big) < \infty. $$
  holds for some~$\alpha_0>2$ and~$\lambda_0, \kappa_0>0$. Suppose that~$\phi$ is not of the form~$c\log{} + f - f\circ T$ for a function~$f:[0,1]\to\R$ and~$c\in\R$. Then, for some~$\sigma>0$ and
  \begin{equation}
    \mu = \frac{12}{\pi^2}\int_0^1\frac{\phi(x)\df x}{1+x},\label{eq:def-mean}
  \end{equation}
  we have
  \begin{equation}
    \P_Q\Big(\frac{S_\phi(x) - \mu \log Q}{\sigma\sqrt{\log Q}} \leq v\Big) = \Phi(v) + O\Big(\frac1{\sqrt{\log Q}} + \frac1{(\log Q)^{\alpha_0/2-1-\eps}}\Big)\label{eq:CLT-convlaw}
  \end{equation}
  uniformly in~$v\in\R$.
\end{corollary}

A variation on the argument shows that the milder hypothesis~$\alpha_0=2$ implies the estimate~\eqref{eq:CLT-convlaw} with a qualitative error term~$o(1)$ as~$Q\to\infty$.
For any~$k\in\N_{>0}$, under the condition~$\alpha_0>k$, a variation on the arguments also leads to an estimate with power-saving error term for~$\E_Q(S_\phi(x)^k)$.

\subsection*{Acknowledgments}
This paper was partially written during a visit of of S. Bettin at the Aix-Marseille University and a visit of S. Drappeau at the University of Genova. The authors thank both Institution for the hospitality and Aix-Marseille University and INdAM for the financial support for these visits. The authors thank B. Vallée for help regarding the references, J. Marklof and B. Borda for discussions, and the anonymous referees for her or his remarks.

S. Bettin is member of the INdAM group GNAMPA and his work is partially supported by PRIN 2017 ``Geometric, algebraic and analytic methods in arithmetic''.

\subsection*{Notations}

For any function~$f(s, t)$ of two real or complex variables, and all~$k,\ell\geq 0$, we let~$\partial_{k,\ell}f := \frac{\partial^{k+\ell}}{\partial s^k \partial t^\ell} f$ whenever this function is defined.

We recall that the semi-norm~$\norm{f}_{(\kappa)}$ is defined in~\eqref{eq:def-holdernorm}. The Landau symbol~$f = O(g)$ means that there is a constant~$C\geq 0$ for which~$\abs{f} \leq C g$ whenever~$f$ and~$g$ are defined. The notation~$f \ll g$ means~$f = O(g)$. If the constant depends on a parameter, say~$\eps$, this is indicated in subscript, \emph{e.g.}~$f = O_\eps(g)$ or~$f \ll_\eps g$.

\section{Lemmas}

\subsection{Hölder constants}

We compile here several facts we will use on the Hölder norms~$\HN{\kappa}{f}$ for~$f\in\Hol^\kappa(\I, \C)$.
\begin{enumerate}
  \item For~$f, g\in\Hol^\kappa$, we have
  \begin{equation}
    \HN{\kappa}{fg} \leq \HN{\kappa}{f}\norm{g}_\infty + \norm{f}_\infty \HN{\kappa}{g}.\label{eq:Ckappa-product}
  \end{equation}
  This follows by writing~$fg(x) - fg(y) = (f(x)-f(y))g(x) + f(y)(g(x)-g(y))$.
  \item For~$g\in\Hol^1$ and~$f\in\Hol^\kappa(g(\I))$, we have~$f\circ g\in\Hol^\kappa$ and
  \begin{equation}
    \HN{\kappa}{f\circ g}\leq \HN{1}{g}^\kappa \HN{\kappa}{f|_{g(\I)}}.\label{eq:Ckappa-composition}
  \end{equation}
  This follows by writing~$\frac{|f(g(x))-f(g(y))|}{|x-y|^\kappa} = \big|\frac{g(x)-g(y)}{x-y}\big|^\kappa \frac{|f(g(x))-f(g(y))|}{|g(x)-g(y)|^\kappa}$.
  \item For~$\lambda\in[\kappa, 1]$ and $f\in\Hol^{\kappa/\lambda}$ real, we have
  \begin{equation}
    \HN{\kappa}{\e^{if}} \leq \HN{{\kappa/\lambda}}{f}^{\lambda}.\label{eq:Ckappa-exp-if}
  \end{equation}
  This follows by writing~$\frac{|\e^{if(x)}-\e^{if(y)}|}{|x-y|^\kappa} = \frac{|\e^{i(f(x)-f(y))}-1|}{|x-y|^\kappa} \leq \frac{|f(x)-f(y)|^\lambda}{|x-y|^\kappa}$.
  \item For~$0< \kappa < \lambda$ and~$f\in\Hol^\lambda$, we have
  \begin{equation}
    \label{eq:Ckappa-morediff}
    \HN{\kappa}{f} \leq \HN{0}{f}^{1-\kappa/\lambda} \HN{\lambda}{f}^{\kappa/\lambda}.
  \end{equation}
  This follows by writing~$\frac{|f(x)-f(y)|}{|x-y|^\kappa} = |f(x)-f(y)|^{1-\kappa/\lambda} \big(\frac{|f(x)-f(y)|}{|x-y|^\lambda}\big)^{\kappa/\lambda}$.
\end{enumerate}

\subsection{Oscillating integrals}

We will require the following analogue of van der Corput's lemma. We let
$$ \I = [0, 1]. $$
\begin{lemma}\label{lem:vdC}
  Let~$\Delta, \kappa>0$. Assume that~$\Psi:\I\to\R$ is~$\CC^1$ with~$\Psi' \geq \Delta$ and that~$\Psi'$ is monotonous on~$\I$. Let~$g\in\Hol^\kappa$. Then
  $$ \int_0^1 g(x)\e^{i\Psi(x)} \df x \ll \frac{\norm{g}_\infty}{\Delta} + \frac{\HN{\kappa}{g}}{\Delta^\kappa} . $$
\end{lemma}
\begin{proof}
  The lemma is obtained by combining the methods of the usual van der Corput Lemma (\cite{Stein1993}, Proposition VIII.2, p.~332), and the bound on Fourier coefficients of a Hölder continuous function (\cite{SteinShakarchi2003}, ex.~15, p.~92); we restrict to mentioning the main steps. We change variables and let
  $$ h(x) = \frac{g\circ\Psi^{-1}(x)}{\Psi'\circ\Psi^{-1}(x)}. $$
  Let~$R = (\Psi(\I)\smallsetminus(\Psi(\I)-\pi)) \cup (\Psi(\I)\smallsetminus(\Psi(\I)+\pi))$. We have
  \begin{align*}
    \int_0^1 g(x)\e^{i\Psi(x)} \df x = {}& \int_{\Psi(\I)} h(x) \e^{ix}\df x \\
    = {}& O\Big(\int_{R} \frac{\norm{g}_\infty\df x}{\Psi'\circ\Psi^{-1}(x)}\Big) - \frac12\int_{\Psi(\I) \cap (\Psi(\I)-\pi)} (h(x+\pi)-h(x))\e^{ix} \df x.
  \end{align*}
  Now, on the one hand,
  \begin{align*}
    h(x+\pi)-h(x) = {}& \frac{g\circ\Psi^{-1}(x+\pi) - g\circ\Psi^{-1}(x)}{\Psi'\circ\Psi^{-1}(x)} + (g\circ\Psi^{-1})(x+\pi)\Big(\frac1{\Psi'\circ\Psi^{-1}(x+\pi)} - \frac{1}{\Psi'\circ\Psi^{-1}(x)}\Big) \\
    \ll {}& \frac{\HN{\kappa}{g} \HN{1}{\Psi^{-1}}^\kappa}{\Psi'\circ\Psi^{-1}(x)} + \|g\|_\infty\abs{\frac{1}{\Psi'\circ\Psi^{-1}(x+\pi)}-\frac{1}{\Psi'\circ\Psi^{-1}(x)}}
  \end{align*}
  by~\eqref{eq:Ckappa-composition} on the first term, and on the other hand,
  $$ \int_{\Psi(\I) \cap (\Psi(\I)-\pi)} \abs{\frac{1}{\Psi'\circ\Psi^{-1}(x+\pi)}-\frac{1}{\Psi'\circ\Psi^{-1}(x)}} \df x \leq\int_R \frac{\df x}{\Psi'\circ\Psi^{-1}(x)} = O(1/\Delta) $$
  by monotonicity. We conclude using~$\HN{1}{\Psi^{-1}}\ll 1/\Delta$.
\end{proof}

\section{Properties of the transfer operator}

In this section and the following ones, all implied constants in the notations~$O(\dotsc)$ and~$\ll$ may depend on~$\alpha_0, \kappa_0, \lambda_0$, $m$ and an upper-bound for the values of~\eqref{eq:bound-norm-rough}--\eqref{eq:bound-Ckappa-rough}. Additional dependences will be indicated in subscript.

\subsubsection*{Definition}

Let
\begin{equation}
  \kappa := \min(\tfrac13, \tfrac12 \kappa_0\lambda_0),\label{eq:def-kappa}
\end{equation}
where we recall that~$\kappa_0$ is the Hölder exponent of~$\vphi$ on each interval~$(\frac1{n+1}, \frac1n)$. For all~$\vt\in\R^d$ and~$s\in\C$ with~$\Re(s)>1$, define an operator~$\H^{(j)}_{s,\vt}$ acting on~$\Hol^\kappa([0, 1], \R)$ by
\begin{align*}
  \H_{s,\vt}^{(j)}[f](x) ={}& \sum_{n\geq 1} \frac{\e^{i\langle\vt,\vphi_j(1/(n+x))\rangle}}{(n+x)^s}f\Big(\frac1{n+x}\Big) \\
  = {}& \sum_{h\in\cH} \e^{i\langle\vt,\vphi_j\circ h(x)\rangle} \abs{h'(x)}^{s/2}(f \circ h)(x).
\end{align*}

When~$\vt=(0, \dotsc, 0)$, this is independent of~$j$, in which case we drop the notation~$(j)$. We abbreviate further
$$ \H_s := \H_{s,0}, \qquad \H := \H_{2,0}. $$
Define the norm and the semi-norm
$$ \norm{f}_0 := \norm{f/\xi}_\infty, \qquad \norm{f}_1 := \HN{\kappa}{f/\xi}. $$
Here we recall that~$\xi(x) = \frac1{\log 2}\frac1{1+x}$. 
We equip~$\Hol^\kappa([0, 1])$ with the norms
$$ \norm{f}_{1,\beta} := \norm{f}_0 + \beta^{-\kappa}\norm{f}_1, \qquad \text{where } \beta>0. $$

For~$\vt\in\R^d$ and~$s\in\C$, $\Re(s)>1$, and~$0\leq j \leq m$, let
$$ \Pi^{(j)}_{s,\vt} := \H^{(j)}_{s,\vt} \dotsb \H^{(1)}_{s,\vt}, $$
with~$\Pi^{(0)}_{s,\vt} = \Id$. In what follows, we will often abbreviate
$$ \Pi_{s,\vt} := \Pi_{s,\vt}^{(m)}, $$
and define
\begin{equation}
  \pdt(x) := \prod_{0\leq j \leq m-1} T^j(x), \qquad g_{s,\vt}(x) := \e^{i\langle\vt, \vphi(x)\rangle} \pdt(x)^{s-2}.\label{eq:def-pdt}
\end{equation}
Note that~$g_{s,\vt}\circ h$ belongs to~$\Hol^{\kappa_0}([0, 1])$ for all~$h\in\cH^m$. Moreover, since~$|T'(x)| = x^{-2}$ whenever~$T$ is differentiable at~$x$, for all~$h\in\cH^m$, we have
$$ |h'| = \prod_{j=0}^{m-1}|T'\circ T^{j-1}|^{-1} = \pdt(x)^2, $$
and therefore
\begin{equation}
  \Pi_{s,\vt} : f \mapsto \H^m[g_{s,\vt}f].\label{eq:def-Pi-st}
\end{equation}

% \subsection{First properties}

\subsection{Properties at the central point}

By~\cite{Broise1996} (section 2.2, Proposition 4.1 and Theorem~4.2), along with the Ruelle-Perron-Frobenius theorem (see~\cite{Mayer1991}), we have that the operator~$\H_{2,0}$ acting on~$\Hol^\kappa([0,1])$ is quasi-compact. It has~$1$ as a simple eigenvalue, and no other eigenvalue of modulus~$\geq 1$. The projection associated with the eigenvalue~$1$ is given by
$$ \P_{2,0}[f](x) = \Big(\int_{[0,1]}f\df\nu\Big) \xi(x), $$
where~$\nu$ is the Lebesgue measure. In particular, we have
$$ \H_{2,0}[\xi] = \xi. $$
Since~$\Pi_{2,0}^{(m)} = \H_{2,0}^m$, we obtain the existence of a linear operator~$\N_{2,0}$ acting on~$\Hol^\kappa$, such that
\begin{equation}
  \Pi_{2,0}^{(m)} = \P_{2,0} + \N_{2,0},\label{eq:pi20-spectral-decomp}
\end{equation}
and additionally the spectral radius of~$\N_{2,0}$ satisfies~$\srd(\N_{2,0})<1$ and~$\P_{2,0} \N_{2,0} = \N_{2,0} \P_{2,0} = 0$.
We will use on several occasions that for~$f$ continuous by parts,
\begin{equation}
  \int_{[0, 1]} \H_{2, 0}[f] \df \nu = \int_{[0,1]} f\df \mu.\label{eq:H-invar-integ}
\end{equation}

\subsection{Dominant spectral properties}

In the sequel, we will repeatedly use the following facts. We let~$\cH^\ast := \cup_{m\geq 0} \cH^m$.
\begin{itemize}
  \item The bounded distortion property~\cite[eq.~(3.1)]{BaladiVallee2005}: for all~$h\in\cH^\ast$, we have~$\abs{h''} \ll \abs{h'}$ with a uniform constant, in particular, independently of the depth of~$h$. This implies that for all~$x\in[0, 1]$ and~$h\in\cH^\ast$,
    \begin{equation}
      \abs{h'(x)} \asymp \abs{h'(0)}.\label{eq:bdd-distort}
    \end{equation}
  \item For all~$q>\tfrac12$, we have
  $$ \sum_{h\in\cH} \norm{h'}_\infty^q < \infty. $$
  \item The ``contracting ratio'' property of the inverse branches~\cite[bound~2.4, fig.~1]{BaladiVallee2005}: for some~$\rho\in[0, 1)$, we have the following uniform bound for all~$\ell\in\N$ and~$h\in\cH^\ell$:
  \begin{equation}
    \HN{1}{h} \ll \rho^\ell. \label{eq:bound-Cgammah}
  \end{equation}
\end{itemize}

\begin{lemma}
  \begin{itemize}
    \item For~$\Re(s) = \sigma>1$, we have
    $$ \Pi_{s,\vt}[f] \leq \Pi_{\sigma,0}[\abs{f}]. $$
    \item For all~$\sigma>1$, there exists~$A_\sigma>0$ such the map~$\sigma\mapsto A_\sigma$ is Lipschitz-continuous and decreasing, $A_2=1$, and we have the bounds on operator norms
    \begin{equation}
      \norm{\H_{\sigma,0}}_0\leq A_\sigma, \qquad \norm{\Pi_{\sigma,0}}_0 \leq A_{\sigma}^m.\label{eq:bound-norms-PiH}
    \end{equation}
    In particular,~$A_\sigma \leq \e^{O(2-\sigma)}$ for~$\sigma \in (1, 2]$. 
    \item For some~$\rho<1$, all~$k\in\N$ and all~$f\in\Hol^k$ with~$f\geq 0$, we have
    \begin{equation}
      \norm{\Pi_{2,0}^k[f]}_0 \ll \int_{[0,1]} f \df \nu + \rho^k \norm{f}_0.\label{eq:rough-spectralbound}
    \end{equation}
    The implied constant is absolute.
  \end{itemize}
\end{lemma}
\begin{proof}
  \begin{itemize}
    \item The first statement is trivial by the triangle inequality.
    \item By a direct computation, we have
    $$ \norm{\H_\sigma}_0 \leq A_\sigma := \sup_{x\in[0,1]} \frac1{\xi(x)}\sum_{n\geq 1} \frac1{(n+x)^{\sigma}}\xi\Big(\frac1{n+x}\Big). $$
    The properties we require of~$A_\sigma$ are readily verified.
    \item The third item follows from~\eqref{eq:pi20-spectral-decomp} with any fixed~$\rho\in(\srd(\N_{2,0}), 1)$, by the definition of the spectral radius.
  \end{itemize}
\end{proof}

\subsection{\texorpdfstring{Spectral gap at~${\boldsymbol t}=0$ and~$\tau\neq 0$}{Spectral gap at~$t=0$ and~$\tau\neq 0$}}

\begin{lemma}\label{lem:bound-norm0-at2}
  For~$\tau\neq 0$, we have~$\norm{\H_{2+i\tau}}_0<1$, and so similarly for~$\Pi_{2+i\tau,0}$.
\end{lemma}
\begin{proof}
  This is a well-studied phenomenon; see~\cite[Proposition~6.1]{ParryPollicott1990}. Our exact statement for~$\H_{2+i\tau}$ acting on a space of holomorphic functions is proved in~\cite[p.~476]{Vallee2003} (see also~\cite[Proposition~3.2]{Morris2015}), however an inspection of the proof shows that it actually holds for~$\H_{2+i\tau}$ acting on~${\mathcal C}([0,1])$, and therefore also on~$\Hol^\kappa$.
\end{proof}

\subsection{Perturbation}

\begin{lemma}\label{lem:perturbationbound-norm}
  For all~$\eps>0$, there exists~$\delta, t_0>0$ such that for~$1\leq j\leq m$, the following holds.
  \begin{itemize}
    \item For~$\sigma\geq 2-\delta$ and~$\nvt\leq t_0$, we have
    \begin{equation}
      \norm{\Pi^{(j)}_{s,\vt} - \Pi^{(j)}_{s,0}}_0 \ll_\eps \nvt + \nvt^{\alpha_0 - \eps}. \label{eq:bound-normdiff-Pi-t}
    \end{equation}
    \item For~$2-\delta<\sigma \leq 3$ and all~$\tau\in\R$, we have
    \begin{equation}
      \norm{\Pi^{(j)}_{\sigma+i\tau,0} - \Pi^{(j)}_{2+i\tau,0}}_0 \ll \abs{\sigma-2}.\label{eq:bound-normdiff-Pi-sigma}
    \end{equation}
    \item For~$\tau_1, \tau_2 \in\R$, we have
    \begin{equation}
      \norm{\Pi^{(j)}_{2+i\tau_1,0} - \Pi^{(j)}_{2+i\tau_2,0}}_0 \ll \abs{\tau_1-\tau_2}. \label{eq:bound-normdiff-Pi-2}
    \end{equation}
  \end{itemize}
\end{lemma}

\begin{proof}
  We assume~$j=m$, the general case being essentially identical.
  \begin{itemize}
    \item We have
    \begin{align*}
      \norm{\Pi_{s,\vt}-\Pi_{s,0}}_0 = {}& \sup_{\norm{f}_0 = 1} \norm{\H^m[(g_{s,\vt}-g_{s,0})f]}_0 \\
      \ll {}& \sum_{h\in\cH^m} \abs{h'(0)} \norm{(g_{s,\vt}-g_{s,0})\circ h}_\infty.
    \end{align*}
    However, for all~$x\in(0,1]$, by the bounded distortion property~\eqref{eq:bdd-distort}, we have
    $$ \abs{g_{s,\vt}(h(x)) - g_{s,0}(h(x))} \ll \abs{h'(0)}^{\sigma/2-1} |\e^{i\langle \vt, \vphi(x)\rangle}-1| \ll \abs{h'(0)}^{\sigma/2-1} (\nvt\norm{\vphi(x)})^\alpha $$
    with~$\alpha = \min(1, \alpha_0-\eps)$. By Hölder's inequality, we deduce
    \begin{align*}
      \norm{\Pi_{s,\vt}-\Pi_{s,0}}_0 \ll {}& \nvt^\alpha \sum_{h\in\cH^m} \abs{h'(0)}^{\sigma/2} \norm{\vphi\circ h}_\infty^\alpha \\
      \ll {}& \nvt^\alpha \Big(\sum_{h\in\cH^m} \abs{h'(0)} \norm{\vphi|_{h(\I)}}_\infty^{\alpha_0} \Big)^{\alpha/\alpha_0} \Big(\sum_{h\in\cH^m} \abs{h'(0)}^{q}\Big)^{1-\alpha/\alpha_0} \numberthis\label{eq:perturb-Pi-t}
    \end{align*}
    with~$q = \frac{\sigma\alpha_0 - 2\alpha}{2(\alpha_0-\alpha)}$. Picking~$\delta = \frac12(1-\frac\alpha{\alpha_0})=O(\eps)$ ensures that~$q>\frac12$, and with our hypothesis~\eqref{eq:bound-norm-rough} and the bounded distortion property~\eqref{eq:bdd-distort}, we obtain
      \begin{align*}
        \sum_{h\in\cH^m} \abs{h'(0)} \norm{\vphi|_{h(\I)}}_\infty^{\alpha_0}
        \ll{}& \sum_{j=1}^m \sum_{h = h_1 \circ \dotsb \circ h_m \in \cH^m} \abs{h_1'(0)\dotsb h_m'(0)} \norm{\vphi_j \circ T^{j-1}|_{h(\I)}}_\infty^{\alpha_0} \\
        \ll{}& \sum_{j=1}^m \sum_{h_j\in \cH} \abs{h_j'(0)} \norm{\vphi_j|_{h_j(\I)}}_\infty^{\alpha_0} < \infty. \numberthis\label{eq:triangleineq-hypoth-phi}
      \end{align*}
    Therefore both sums in~\eqref{eq:perturb-Pi-t} are bounded in terms of~$\eps$ and~$\vphi$ only, and we deduce the claimed bound~$\norm{\Pi_{s,\vt} - \Pi_{s,0}} \ll_\eps \nvt^\alpha$.
  \item Proceeding as above, we find
    \begin{align*}
      \norm{\Pi_{\sigma+i\tau,0}-\Pi_{2+i\tau,0}}_0 {}& \ll \sum_{h\in\cH^m} \abs{h'(0)} \norm{(\pdt^{\sigma-2}-1)\circ h}_\infty \\
      {}& \ll \abs{\sigma-2} \sum_{h\in\cH^m} (1+\abs{\log\abs{h'(0)}})\abs{h'(0)}^{\max(\sigma/2, 1)}.
    \end{align*}
    For any~$\sigma>1$, the last sum is finite, so that our statement follows for any fixed~$\delta\in(0,1)$.
    \item Once again proceeding again as above, we have
    $$ \norm{\Pi_{2+i\tau_1,0} - \Pi_{2+i\tau_2,0}}_0 \ll \sum_{h\in\cH^m} \abs{h'(0)} \norm{(\pdt^{i(\tau_1-\tau_2)}-1)\circ h}_\infty. $$
    Letting~$\tau=\tau_1-\tau_2$, we insert the inequality~$\norm{(\pdt^{i\tau}-1)\circ h}_\infty \ll \abs{\tau}(1 + \abs{\log\abs{h'(0)}})$. The resulting sum over~$h$ being absolutely bounded, we deduce~$\norm{\Pi_{2+i\tau_1,0}-\Pi_{2+i\tau_2,0}}_0 \ll \abs{\tau}$ as required.\qedhere
  \end{itemize}
\end{proof}

\subsection{\texorpdfstring{First estimate on~$\norm{\Pi_{s,\vt}}_1$}{First estimate on~$\norm{\Pi_{s,t}}_1$}}

The following is a weak form of~\cite[Lemma~2]{BaladiVallee2005} (which is refered to, there, as a Lasota-Yorke type inequality).
We recall that the Hölder exponent~$\kappa$ was defined in~\eqref{eq:def-kappa}.

\begin{lemma}\label{lem:bound-norm1}
  For all~$\delta\in(0, \kappa)$, there exists~$\rho\in[0,1)$ such that for~$\sigma\geq 2-\delta$, $\vt\in\R^d$ with~$\nvt\leq 1$, and~$\tau\in\R$, we have
  \begin{equation*}
    \norm{\Pi_{s,\vt}[f]}_1 \leq O(1 + \abs{s}^\kappa) \norm{f}_0 + \rho\norm{f}_1, 
  \end{equation*}
\end{lemma}
\begin{proof}
  Let~$f\in\Hol^\kappa$. We write
  $$ \tfrac1\xi \Pi_{s,\vt}[f] = \sum_{h\in\cH^m} \tfrac1\xi (\xi\circ h) \abs{h'}^{s/2} \e^{i\langle{\vt,\vphi\circ h}\rangle} (\tfrac f\xi \circ h). $$
  Splitting as a sum of differences, we obtain
  $$ \HN{\kappa}{\tfrac1\xi \Pi_{s,\vt}[f]} \leq \norm{f}_0 \sum_{h\in\cH^m} \HN{\kappa}{\tfrac1\xi (\xi\circ h) \abs{h'}^{s/2} \e^{i\langle{\vt,\vphi\circ h}\rangle}} + \rho_\sigma\norm{f}_1, $$
  where
  $$ \rho_\sigma := \norm{\sum_{h\in\cH^m} \tfrac1\xi (\xi\circ h) \abs{h'}^{\sigma/2} w_h }_\infty, \qquad w_h(x) = \sup_{0\leq y \leq 1} \abs{\frac{h(x)-h(y)}{x-y}}^\kappa = \abs{h'(x)h'(0)}^{\kappa/2}. $$
  This last equality follows from the fact that each~$h\in\cH^\ast = \cup_{m\geq 0} \cH^m$ is a homography associated with an element of~$GL_2(\Z)$ with non-negative entries, $h(x) = \frac{ax + b}{cx+d}$, and so the supremum above is~$\sup_y |(cx+d)(cy+d)|^{-1} = |h'(0)h'(x)|^{-1/2}$. Since~$\abs{h'(0)}\leq 1$ by the chain rule, we deduce
  $$ \sum_{h\in\cH^m} \tfrac1\xi (\xi\circ h) \abs{h'}^{\sigma/2} w_h \leq \sum_{h\in\cH^m} \tfrac1\xi (\xi\circ h) \abs{h'} \abs{h'}^{(\sigma-2+\kappa)/2}. $$
  Note that~$\sigma-2+\kappa \geq \kappa-\delta>0$ by hypothesis, and we have~$\abs{h'}\leq 1$ by the chain rule. Moreover, for any value of~$m$, we may find at least one~$h\in\cH^m$ with~$\norm{h'}_\infty < 1$, \emph{e.g.} by composing repeatedly~$t\mapsto \frac1{2+t}$. Since~$\frac1\xi (\xi\circ h) \abs{h'} > 0$, we deduce
  $$ \rho_{\kappa-\delta} < \norm{\sum_{h\in\cH^m} \tfrac1\xi (\xi\circ h) \abs{h'}}_\infty = 1. $$
  Next, by using the rules~\eqref{eq:Ckappa-composition}, \eqref{eq:Ckappa-exp-if}, \eqref{eq:Ckappa-morediff}, the bounded distortion property~$\abs{h''}\ll \abs{h'}$, and simple computations, we obtain successively
  \begin{align*}
    \HN{\kappa}{\tfrac1\xi} {}& \ll 1,  & \HN{\kappa}{\abs{h'}^{s/2}} {}& \ll \abs{s}^\kappa \norm{h'}_\infty^{\sigma/2}, \\
    \HN{\kappa}{\xi \circ h} {}& \leq \norm{h'}_\infty^\kappa, & \HN{\kappa}{\e^{i\langle{\vt,\vphi\circ h}\rangle}} {}& \ll \norm{h'}_\infty^{\kappa} \nvt^{\kappa/\kappa_0} \HN{{\kappa_0}}{\vphi}^{\kappa/\kappa_0}.
  \end{align*}
  In the last line, we used the definition~\eqref{eq:def-kappa}. Grouping these bounds using~\eqref{eq:Ckappa-product}, we deduce
  \begin{align*}
    \sum_{h\in\cH^m} \HN{\kappa}{\tfrac1\xi (\xi\circ h) \abs{h'}^{s/2} \e^{i\langle{\vt,\vphi\circ h}\rangle}}
    {}& \ll 1 + \abs{s}^\kappa + \sum_{h\in\cH^m} \norm{h'}_\infty^{\sigma/2+\kappa} \HN{{\kappa_0}}{\vphi}^{\kappa/\kappa_0} \\
    {}& \ll 1 + \abs{s}^\kappa
  \end{align*}
  since~$\sigma/2+\kappa\geq 1 +\kappa/2 \geq 1$,~$\HN{{\kappa_0}}{\vphi}^{\kappa/\kappa_0} \leq 1 + \HN{{\kappa_0}}{\vphi}^{\lambda_0}$ (by the definition~\eqref{eq:def-kappa}), and by our hypothesis~\eqref{eq:bound-Ckappa-rough}.
\end{proof}

\section{Meromorphic continuation}

Following~\cite{Vallee2000}, define the generating series
\begin{equation*}
  \gS(\vt, s) :=  \sum_{x\in\Q\cap(0,1]} q(x)^{-s} \exp(i\langle{\vt,S_\vphi(x)}\rangle), 
\end{equation*}
where~$q(x)$ is the reduced denominator of~$x$.

\begin{lemma}\label{lem:S-Pi}
  For~$\Re(s)>2$ and~$\vt\in\R^d$, we have
  $$ \gS(\vt, s) = (\Pi^{(0)}_{s,\vt} + \Pi^{(1)}_{s,\vt} + \dotsb + \Pi^{(m-1)}_{s,\vt})(\Id - \Pi_{s,\vt})^{-1}[\1](1). $$
\end{lemma}

\begin{proof}
  This is a straightforward extension of the computations in~\cite[eq.~(2.17)]{BaladiVallee2005}.
\end{proof}

The aim of this section is to show the meromorphic continuation of~$\gS(s, t)$ to a half-plane~$\Re(s)\geq 2-\delta$. This will will then be used in conjunction with Mellin transformation to prodive an estimate with power-saving for~$\E_Q(\e^{i\langle \vt, S(x)\rangle})$.
We use different arguments according to the size of~$\Im(s)$, which is refered to as the height in what follows.

\subsection{Small height}

The following lemma deals with~$\tau$ in a neighborhood of~$0$. In this case we use the spectral expansion of~$\Pi_{s, \vt}$.

\begin{lemma}\label{lem:small-height}
  There exists~$\delta, \tau_0, t_0>0$, such that for all~$\sigma, \tau, t\in\R$ with~$\abs{\sigma-2}\leq \delta$, $\abs{\tau}\leq \tau_0$ and~$\nvt\leq t_0$, the operator~$\Pi_{s,\vt}$ acting on~$(\Hol^\kappa, \norm{\cdot}_{1, 1})$ is quasi-compact, and for some~$\lambda(s,\vt)\in\C$, we have
  $$ \Pi_{s,\vt} = \lambda(s,\vt) \P_{s,\vt} + \N_{s,\vt} $$
  where~$\P_{s,\vt}$ is of rank~1,~$\P_{s,\vt}\N_{s,\vt}=\N_{s,\vt}\P_{s,\vt}=0$, $\P_{s,\vt}^2=\P_{s,\vt}$, and~$\srd(\N_{s,\vt}) < 1-\delta$. Moreover, for each such fixed~$\vt$, the operators~$\Pi_{s,\vt}$, $\P_{s,\vt}$, $\N_{s,\vt}$ and the eigenvalue~$\lambda(s, \vt)$ depend analytically on~$s$.
\end{lemma}

\begin{proof}
  This is a direct application of \emph{e.g.} Theorem~2.3 of~\cite{Kloeckner2019}; see also chapter~IV.3 of~\cite{Kato}, and~\cite[p.342]{BaladiVallee2005}.
\end{proof}

\subsection{Moderate height}

The following lemma is concerned with~$\tau$ of bounded size and away from~$0$. The main tool is again perturbation theory.

\begin{lemma}\label{lem:moderate-height}
  For all~$\tau_0, \tau_1 >0$ with~$\tau_0 < \tau_1$, there exists~$\delta, t_0>0$ such that for all $t, \sigma, \tau\in\R$ with~$\nvt\leq t_0$, $\sigma\geq 2-\delta$ and~$\tau_0 \leq \abs{\tau}\leq \tau_1$, we have
  $$ \norm{\Pi_{s,\vt}}_0 \leq 1-\delta. $$
\end{lemma}

\begin{proof}
  By~\eqref{eq:bound-normdiff-Pi-2} and the triangle inequality, the map~$\tau\mapsto \norm{\Pi_{2+i\tau,0}}_0$ is continuous. By Lemma~\ref{lem:bound-norm0-at2}, we deduce that for some number~$\eta>0$, depending on~$\tau_0$ and $\tau_1$, we have the bound~$\norm{\Pi_{2+i\tau,0}}_0 \leq 1-\eta$. By the perturbation bounds~\eqref{eq:bound-normdiff-Pi-t} and \eqref{eq:bound-normdiff-Pi-sigma}, we may pick~$\delta, t_0$ small enough so that~$\norm{\Pi_{s,\vt}-\Pi_{2+i\tau,0}}_0 \leq \eta/2$, and our claim follows.
\end{proof}

\subsection{Large height}

The following lemma deals with the case of large enough~$\abs{\tau}$.

\begin{lemma}\label{lem:large-height}
  For some constants~$\delta, \tau_1, C>0$, whenever~$\sigma \geq 2-\delta$, $\vt\in\R^d$ with~$\nvt\leq 1$, and~$\abs{\tau}\geq \tau_1$, the operator~$\Pi_{s,\vt}$ acting on~$\Hol^\kappa$ has spectral radius~$\srd(\Pi_{s,\vt})<1$, and
  $$ \sum_{j\geq 0} \big\|\Pi_{s,\vt}^j\big\|_{1,\tau} \ll \abs{\tau}^{C\abs{\sigma-2}} \log \abs{\tau}. $$
\end{lemma}

For unbounded values of~$\tau$, perturbation theory is not effective, instead we adapt the arguments of Dolgopyat~\cite{Dolgopyat1998} and Baladi-Vallée~\cite{BaladiVallee2005}, which exploit the cancellation due to the varying argument of~$|h'|^{i\tau}$. Compared with Baladi-Vallée's setup, we make two modifications: we work with Hölder-continuous functions, rather than~$\CC^1$, and the cost function is not assumed to be constant on each interval of the partition.

\subsubsection{Sums over branches}

We will require two estimates involving sums over inverse branches on~$T$. Define, as in~\cite[eq.~(3.10)]{BaladiVallee2005},
$$ \Delta(h_1, h_2) := \inf_{x\in[0,1]}\abs{\frac{h_1''(x)}{h_1'(x)} - \frac{h_2''(x)}{h_2'(x)}}, $$
Note that by the bounded distortion property, there exists~$\Delta_+ \geq 1$ such that
\begin{equation}
  \Delta(h_1, h_2) \leq \Delta_+\label{eq:upperbound-Delta}
\end{equation}
for all~$h_1, h_2 \in\cH^\ast$. The following property is the statement that condition UNI.(a) of Baladi-Vallée~\cite{BaladiVallee2005} holds for the Gauss map.
\begin{lemma}\label{lem:sumbranch-delta}
  For some absolute constant~$\rho\in[0,1)$, we have uniformly in~$n\in\N$, $h_1\in\cH^n$ and~$u\in[0,\Delta_+]$ that
  $$ S(u) := \ssum{h_2 \in \cH^n \\ \Delta(h_1,h_2)\leq u} \abs{h_2'(0)} \ll \rho^n + u. $$
\end{lemma}
\begin{proof}
  See Lemmas~1 and~6 of~\cite{BaladiVallee2005}; the main point is the construction of a dual dynamical system (Section 3.4 of~\cite{BaladiVallee2005}) which encodes naturally the quantity~$\Delta(h_1,h_2)$, and satisfies the dominant spectral bound~\eqref{eq:rough-spectralbound}. The dual map of the Gauss map is in fact the Gauss map.
\end{proof}

\begin{lemma}\label{lem:sumbranch-ckappa}
  Under the assumption~\eqref{eq:bound-Ckappa-rough}, uniformly for all~$n\in\N_{>0}, 0\leq j \leq n-m$, we have
  $$ \sum_{h\in\cH^n} \abs{h'(0)} \HN{\kappa_0}{\vphi|_{T^j\circ h(\I)}}^{\lambda_0} \ll 1. $$
\end{lemma}
We recall that~$\I = [0, 1]$.

\begin{proof}
  We decompose~$h = h_1 \circ h_2 \circ h_3$, where~$h_1\in\cH^j$, $h_2\in\cH^m$ and~$h_3\in\cH^{n-j-m}$. We have
  $$ \abs{h'(0)} \ll \abs{h_1'(0)h_2'(0)h_3'(0)}, \qquad \HN{\kappa_0}{\vphi|_{T^j\circ h(\I)}} \leq \HN{{\kappa_0}}{\vphi|_{h_2(\I)}}, $$
  so that
  $$ \sum_{h\in\cH^n} \abs{h'(0)} \HN{{\kappa_0}}{\vphi|_{T^j\circ h(\I)}}^{\lambda_0} \ll \Big(\sum_{h_1\in\cH^j} \abs{h_1'(0)}\Big) \Big(\sum_{h_2\in\cH^m} \abs{h_2'(0)} \HN{{\kappa_0}}{\vphi_{h_2(\I)}}^{\lambda_0}\Big) \Big(\sum_{h_3\in\cH^{n-j-m}} \abs{h_3'(0)}\Big). $$
  The sums over~$h_1$ and~$h_3$ are uniformly bounded by~\eqref{eq:bound-norms-PiH}. The sum over~$h_2$ is finite by our hypothesis~\eqref{eq:bound-Ckappa-rough} and the triangle inequality (\emph{cf.}~\ref{eq:triangleineq-hypoth-phi}).
\end{proof}

\subsubsection{Bound on the $L^2$ norm}

We recall that~$A_\sigma$ is an upper-bound for the norm~$\|H_{\sigma, 0}\|_{0}$ provided in~\eqref{eq:bound-norms-PiH}.

\begin{lemma}\label{lem:bound-l2-norm}
  For some~$\delta, t_0>0$ and~$\rho\in[0, 1)$, whenever $\abs{\sigma-2}\leq \delta$,~$\abs{\tau}\geq 1$,~$\nvt\leq t_0$ and~$\ell\in\N$, we have
  $$ \Big(\int_{[0,1]} \abs{\Pi_{s,\vt}^{\ell}[f]}^2 \df\nu \Big)^{1/2} \ll A_{2\sigma-2}^{m\ell/2} \big(\big(\abs{\tau}^{-\kappa/2} + \rho^{m\ell/4}\big)\norm{f}_0 + \rho^{\kappa m\ell/2}\abs{\tau}^{-\kappa/2}\norm{f}_1\big). $$
\end{lemma}
\begin{proof}
  Changing~$f$ to~$\bar{f}$, $\vt$ to~$-\vt$ and taking conjugates if necessary, we may assume that~$\tau\geq 0$.
  Define, for all~$\ell\in\N_{>0}$, $\vpsi_\ell := \sum_{0\leq j < \ell} \vphi \circ T^{mj}$,
  so that
  $$ \Pi_{s,\vt}^\ell[f] = \sum_{h\in\cH^{m\ell}} \e^{i\langle{\vt,\vpsi_\ell\circ h}\rangle} \abs{h'}^{s/2} (f \circ h). $$
  We note that for all~$h\in\cH^{m\ell}$, by~\eqref{eq:Ckappa-composition} and \eqref{eq:Ckappa-exp-if}, we have
  \begin{equation}\label{sec:bound-eit-psi}
    \begin{aligned}
      \HN{\kappa}{\e^{i\langle{\vt, \vpsi_\ell \circ h}\rangle}} \leq {}& \sum_{0\leq j < \ell} \HN{1}{T^{mj} \circ h}^\kappa \HN{\kappa}{\e^{i\langle \vt, \vphi \rangle}|_{T^{mj}\circ h(\I)}} \\
      \ll {}& \sum_{0\leq j < \ell} \rho^{m(\ell-j)\kappa} \HN{{\kappa_0}}{\vphi|_{T^{mj}\circ h(\I)}}^{\kappa/\kappa_0}.
    \end{aligned}
  \end{equation}
  For~$h_1, h_2 \in\cH^{m\ell}$, let
  $$ g_{h_1,h_2} := \e^{i\langle{\vt,\vpsi_\ell\circ h_1 - \vpsi_\ell\circ h_2}\rangle} \abs{h_1' h_2' }^{\sigma/2} (f\circ h_1)\overline{(f\circ h_2)}. $$
  This defines a function in~$\Hol^\kappa$. Expanding the square, we find
  $$ \int_{[0,1]} \abs{\Pi_{s,\vt}^\ell[f]}^2 \df\nu = \sum_{h_1, h_2\in\cH^{m\ell}} I(h_1, h_2), \qquad I(h_1, h_2) := \int_0^1 g_{h_1,h_2}(x) \abs{\frac{h_1'(x)}{h_2'(x)}}^{i\tau/2} \df x. $$
  We have, for all~$h_1, h_2\in\cH^{m\ell}$, the trivial bound
  \begin{equation}
    \abs{I(h_1, h_2)} \ll \norm{g_{h_1,h_2}}_\infty.\label{eq:bounds-Ihh-triv}
  \end{equation}
  On the other hand, for all~$h_1, h_2\in\cH^{m\ell}$ satisfying~$\Delta(h_1, h_2)>0$, we have from Lemma~\ref{lem:vdC} the bound
  \begin{equation}
    \abs{I(h_1, h_2)} \ll \frac{\norm{g_{h_1,h_2}}_\infty}{\abs{\tau}\Delta(h_1,h_2)} + \frac{\HN{\kappa}{g_{h_1,h_2}}}{(\abs{\tau} \Delta(h_1,h_2))^\kappa},\label{eq:bound-Ihh-vdc}
  \end{equation}
  The norms are bounded, using~\eqref{eq:bound-Cgammah},~\eqref{eq:Ckappa-product}, \eqref{eq:Ckappa-composition} and~\eqref{sec:bound-eit-psi}, by
  \begin{align}
    \norm{g_{h_1,h_2}}_\infty \ll {}& \norm{f}_\infty^2 \abs{h_1'(0) h_2'(0)}^{\sigma/2}, \label{eq:bounds-ghh} \\
    \HN{\kappa}{g_{h_1,h_2}} \ll {}& \norm{f}_\infty \abs{h_1'(0) h_2'(0)}^{\sigma/2}\big\{(1 + \sum_{h\in\{h_1,h_2\}}\sum_{0\leq j < \ell} \rho^{m(\ell-j)\kappa} \HN{{\kappa_0}}{\vphi|_{T^{mj}\circ h(\I)}}^{\kappa/\kappa_0}) \norm{f}_\infty \notag \\
    {}& \hspace{13em} + \rho^{\kappa m\ell} \HN{\kappa}{f}\big\}. \notag
  \end{align}
  We write~$\HN{\kappa}{f}\norm{f}_\infty \ll \HN{\kappa}{f}^2 + \norm{f}^2_\infty \ll \HN{\kappa}{f/\xi}^2 + \norm{f/\xi}^2_\infty$, which implies the variant
  \begin{equation}
    \HN{\kappa}{g_{h_1,h_2}} \ll \abs{h_1'(0) h_2'(0)}^{\sigma/2}\Big\{\Big(1 + \sum_{h\in\{h_1,h_2\}}\sum_{0\leq j < \ell} \rho^{m(\ell-j)\kappa} \HN{{\kappa_0}}{\vphi|_{T^{mj}\circ h(\I)}}^{\kappa/\kappa_0}\Big)\norm{f}_0^2 + \rho^{\kappa m\ell} \norm{f}_1^2\Big\}.\label{eq:bound-ghh-var}
  \end{equation}
  Next, for all~$u\in[0, \Delta_+]$ (where we recall~\eqref{eq:upperbound-Delta}), we have uniformly
  \begin{align*}
    K(u) := {}& \max_{0\leq j < \ell} \ssum{h_1, h_2 \in \cH^{m\ell} \\ \Delta(h_1,h_2)\leq u} \abs{h_1'(0)h_2'(0)}^{\sigma/2}(1+\HN{{\kappa_0}}{\vphi|_{T^{mj}\circ h_1(\I)}}^{\lambda_0/2}) \\
    \ll {}& \max_{0\leq j < \ell} \Big(\sum_{h_1, h_2\in \cH^{m\ell}} \abs{h_1'(0)}(1+\HN{{\kappa_0}}{\vphi|_{T^{mj}\circ h_1(\I)}}^{\lambda_0})\abs{h_2'(0)}^{\sigma-1}\Big)^{1/2} \\
    {}& \hspace{6em} \times \Big(\ssum{h_1, h_2\in \cH^{m\ell} \\ \Delta(h_1, h_2)\leq u} \abs{h_1'(0)}^{\sigma-1} \abs{h_2'(0)} \Big)^{1/2} \\
    \ll {}& A_{2\sigma-2}^{m\ell} (\rho^{m\ell/2} + u^{1/2})
  \end{align*}
  by Lemmas~\ref{lem:sumbranch-delta} and \ref{lem:sumbranch-ckappa}.
  Let~$\eta\in(0, 1]$ be a parameter. We insert the bounds~\eqref{eq:bounds-ghh} and~\eqref{eq:bound-ghh-var} in~\eqref{eq:bounds-Ihh-triv}, \eqref{eq:bound-Ihh-vdc}, and we sum over~$(h_1, h_2)$. When~$\Delta(h_1, h_2)\leq \eta$, we use the trivial bound~\eqref{eq:bounds-Ihh-triv}, otherwise we use~\eqref{eq:bound-Ihh-vdc}. Using our bond on~$K(u)$ above, the symmetry~$h_1\leftrightarrow h_2$, the fact that~$\kappa/\kappa_0 \leq \lambda_0/2$, and partial summation, we find
  \begin{align*}
    A_{2\sigma-2}^{-m\ell}\sum_{h_1, h_2\in\cH^{m\ell}} I(h_1, h_2)
    \ll {}& \norm{f}_0^2\Big(K(\eta) + \frac{K(\Delta_+)}{\tau^\kappa} + \int_\eta^{\Delta_+} \Big(\frac1{\tau u} + \frac{\kappa}{(\tau u)^\kappa}\Big) \frac{K(u)\df u}{u}\Big) \\
    {}& \hspace{1em} + \frac{\rho^{\kappa m\ell}\norm{f}_1^2}{\tau^\kappa}\Big(K(\Delta_+) + \kappa\int_\eta^{\Delta_+} \frac{K(u)\df u}{u^{\kappa+1}} \Big) \\
    \ll {}& \norm{f}_0^2 \Big((\rho^{m\ell/2}+\eta^{1/2})\Big(1 + \frac1{\tau\eta} + \frac1{(\tau\eta)^\kappa}\Big) + \frac1{\tau^\kappa}\Big) \\
    {}& \hspace{1em} + \norm{f}_1^2 \rho^{\kappa m\ell} \Big(\frac{1}{\tau^\kappa} + \frac{\rho^{m\ell/2}+\eta^{1/2}}{(\tau\eta)^\kappa}\Big).
  \end{align*}
  Choosing~$\eta = 1/\tau$, we obtain
  \begin{align*}
    \Big(\int_{[0,1]} \abs{\Pi_{s,\vt}^{\ell}[f]}^2 \df\nu\Big)^{1/2} \ll {}& A_{2\sigma-2}^{m\ell}\big(\big(\tau^{-\kappa/2} + \rho^{m\ell/4}\big)\norm{f}_0 + \tau^{-\kappa/2}\rho^{\kappa m\ell/2}\norm{f}_1\big)
  \end{align*}
  as claimed.
\end{proof}

\subsubsection{Bound on the~$L^\infty$ norm}

Next, we transfer the $L^2$ bound relative to the invariant measure into an~$L^\infty$ bound, following ideas of Dolgopyat~\cite{Dolgopyat1998} adapted to this context by Baladi and Vallée~\cite[section~3.3]{BaladiVallee2005}.

\begin{lemma}\label{lem:bound-linfty}
  For some~$\rho\in[0,1)$ and~$c_0, \delta, \tau_0>0$, depending on~$\eta$ and~$\kappa$ at most, whenever
  $$ \sigma \geq 2 - \delta, \qquad \tau\geq \tau_0, \qquad \nvt\leq 1, $$
  then letting~$n = \floor{c_0 \log\tau}$, we have
  $$ \norm{\Pi_{s,\vt}^n[f]}_0 \leq \rho^n \norm{f}_{1,\tau}. $$
\end{lemma}
\begin{proof}
  By using the Cauchy--Shwarz inequality, as in~\cite[Lemma~1]{BaladiVallee2005}, for all~$x\in[0,1]$ and~$f\in\Hol^\kappa$, we have
  \begin{align*}
    \frac{\abs{\Pi_{s,\vt}^k[f](x)}}{\xi(x)} \leq {}& 
    \Big(\sum_{h\in\cH^{mk}} \abs{h'(x)}^{\sigma-1} \Big)^{1/2} \Big(\sum_{h\in\cH^{mk}} \abs{h'(x)} (\abs{f}^2\circ h)(x) \Big)^{1/2} \\
    \ll {}& A_{2\sigma-2}^{mk/2} \Big(\int_0^1 \abs{f}^2\df\nu + \rho_1^{mk} \norm{f}_0^2 \Big)^{1/2}
  \end{align*}
  for some~$\rho_1\in[0,1)$ independent of~$k$. We use this with~$f$ replaced by~$\Pi_{s,\vt}^\ell[f]$, with~$k, \ell$ being any choice with~$k+\ell=n$ and~$\ell = k + O(1)$. By Lemma~\ref{lem:bound-l2-norm} and the bound~\eqref{eq:bound-norms-PiH}, we deduce that for all small enough~$\delta\geq 0$, if~$\sigma\geq 2-\delta$, then
  $$ \norm{\Pi_{s,\vt}^{n}[f]}_0 \ll \e^{O(\delta n)}\big((\tau^{-\kappa/2} + \rho_1^{mn/8})\norm{f}_0 + \rho_1^{mn\kappa/4}\tau^{-\kappa/2} \norm{f}_1\big).$$
  By choosing~$n = c \log \abs{\tau} + O(1)$ with~$c = 4(m\abs{\log \rho_1})^{-1}$, and then~$\tau_0$ large enough and~$\delta$ small enough in terms of~$\kappa$, we may ensure that
  $$ \norm{\Pi_{s,\vt}^{n}[f]}_0 \leq \rho^n\norm{f}_{1,\tau}, $$
  with~$\rho = \rho_1^{m\kappa/10} < 1$ and as claimed.
\end{proof}

\subsubsection{Proof of Lemma~\ref{lem:large-height}}

Iterating the bound of Lemma~\ref{lem:bound-norm1}, and using~\eqref{eq:bound-norms-PiH}, we have for~$\delta$ small enough and all~$n\geq 0$ the bound
$$ \norm{\Pi_{s,\vt}^n[f]}_1 \ll \e^{O(\delta n)} \abs{\tau}^\kappa \norm{f}_0 + \rho^n\norm{f}_1, $$
for some~$\rho\in[0,1)$. We replace~$f$ by~$\Pi_{s,\vt}^n[f]$ and use Lemmas~\ref{lem:bound-linfty} and~\ref{lem:bound-norm1}. We find that for some constants~$\tau_0\geq 0$, $c_0>0$ and~$\rho\in[0,1)$, if~$\delta$ is small enough and~$n = \floor{c_0 \log\tau}$, then
\begin{align*}
  \norm{\Pi_{s,\vt}^{2n} f}_{1,\tau}
  \ll {}& (\rho^n + \e^{O(\delta n)})\norm{\Pi_{s,\vt}^n f}_0 + \abs{\tau}^{-\kappa}\rho^n \norm{\Pi_{s,\vt}^n f}_1 \\
  \ll {}& (\rho^{2n} + \e^{O(\delta n)}\rho^n)\norm{f}_0 + \abs{\tau}^{-\kappa}(\rho^{2n} + \e^{O(\delta n)}\rho^n)\norm{f}_1.
\end{align*}
At the cost of choosing~$c_0$ large enough and~$\delta$ small enough in terms of the implied constants, we obtain
$$ \norm{\Pi_{s,\vt}^{2n} f}_{1,\tau} \leq \rho^{n/2}\norm{f}_{1,\tau}. $$
By iterating, this bounds also holds for~$n = k\floor{c_0 \log \tau}$, $k\in\N$, from which we deduce by Gelfand's inequality that $\srd(\Pi_{s,\vt})\leq \rho^{1/4}$, and~$\norm{(\Id-\Pi_{s,\vt}^{2n})^{-1}}_{1,\tau} \ll 1$. Finally, from the bounds
$$ \norm{(\Id-\Pi_{s,\vt})^{-1}}_{1,\tau} \leq \norm{(\Id-\Pi_{s,\vt}^{2n})^{-1}}_{1,\tau} \sum_{0\leq j < 2n} \norm{\Pi_{s,\vt}^j}_{1,\tau} $$
and~$\norm{\Pi_{s,\vt}^j}_{1,\tau} \ll \e^{O(\abs{\sigma-2}j)}$, we get the claimed result.

\subsection{Deduction of the meromorphic continuation}

\begin{proposition}\label{prop:mero-continuation}
  For some~$\tau_0, t_0, \delta>0$, and all~$\nvt\leq t_0$, the function~$s\mapsto \gS(s,\vt)$, initially only defined for~$\Re(s)>2$, has a meromorphic continuation to the set
  $$ H := \Big\{s\in\C,\ s = \sigma + i\tau,\ \sigma \geq 2-\delta \}, $$
  with possible poles occuring only for~$\abs{\tau} < \tau_0$ and~$\lambda(s, t) = 1$.
  The meromorphic continuation of~$s\mapsto \gS(s, \vt)$ is bounded uniformly in~$\Re(s)\geq 2-\delta$ and~$\abs{\tau}\geq \tau_0$ by
  \begin{equation}
    \abs{\gS(s, \vt)} \ll \abs{\tau}^{O(\max(0, 2-\sigma))}\log(\abs{\tau}+2).\label{eq:bound-gS-high}
  \end{equation}

  More precisely, for~$\abs{\tau} \leq \tau_0$, the function
  \begin{equation}
    s \mapsto \gS(s,\vt) -  \frac{\lambda(s,\vt)}{1-\lambda(s,\vt)}(\Pi^{(0)}_{s,\vt} + \Pi^{(1)}_{s,\vt} + \dotsb + \Pi^{(m-1)}_{s,\vt})\P_{s,\vt}[\1](1)\label{eq:extractpole-S}
  \end{equation}
  has an analytic continuation to~$\sigma\geq 2-\delta$ and~$\abs{\tau}\leq \tau_0$, and is uniformly bounded there.
\end{proposition}

\begin{proof}
  By Lemma~\ref{lem:S-Pi}, we have for~$\Re(s)>1$
  \begin{equation}
    \gS(s, \vt) = (\Pi_{s, \vt}^{(0)} + \dotsb + \Pi_{s,\vt}^{(m-1)})\sum_{j\geq 0} \Pi_{s,\vt}[\1](0).\label{eq:recall-sum-gS}
  \end{equation}
  Then Lemma~\ref{lem:small-height} yields the analytic continuation of~\eqref{eq:extractpole-S} for~$|\tau|\leq \tau_0$ for some~$\tau_0>0$. Then Lemma~\ref{lem:large-height} ensures the existence of~$\tau_1>0$ such that the sum over~$j$ in~\eqref{eq:extractpole-S} converges uniformly over compacts in the region~$\sigma\geq 2-\delta$ and~$\abs{\tau} \geq \tau_1$, and yields the bound~\eqref{eq:bound-gS-high}. Finally, applying Lemma~\ref{lem:moderate-height} with the values of~$\tau_0$ and~$\tau_1$ gives the same conclusion for~$\tau_0 \leq \abs{\tau} \leq \tau_1$. The conjunction of these three cases covers the whole half-plane~$H$.
\end{proof}

\section{Asymptotic behaviour of the leading eigenvalue}\label{sec:asympt-behav-lead}

In this section, we study the behaviour, for small~$t$ and~$s$ close to~$2$, of the leading eigenvalue~$\lambda(s,\vt)$. The estimates in this section will reduce the problem to the estimation as~$\vt\to 0$ of the integral
$$ \int_0^1 \e^{i\langle \vt, \vphi(x) \rangle} \xi(x) \df x, $$
where we recall the notation~\eqref{eq:def-vgamtot}. We recall the hypotheses~\eqref{eq:bound-norm-rough}, \eqref{eq:bound-Ckappa-rough}.

\subsection{Perturbation theory and existence}

Let
\begin{equation}
  \fd := -m \int_0^1 \log(x) \xi(x) \df x = \frac{m\pi^2}{12\log 2}.\label{eq:def-fd}
\end{equation}

\begin{lemma}\label{lem:approx-partialder}
  For all small enough~$\eps>0$, there exists~$t_0>0$ such that whenever~$\abs{s-2}\leq\eps$ and~$\nvt\leq t_0$, we have
  \begin{equation}
    \partial_{10}\lambda(s,\vt) = -\fd + O(\eps),\label{eq:approx-partial10lambda}
  \end{equation}
  and
  \begin{equation}
    \lambda(s,\vt) - 1 = \big({-\fd} + O(\eps)\big) (s-2) + O(\eps).\label{eq:taylor1-lambda}
  \end{equation}
\end{lemma}
\begin{proof}
  Let~$f_{s,\vt} = \P_{s,\vt}[\xi]$ denote an eigenfunction of~$\Pi_{s,\vt}$ associated with the eigenvalue~$\lambda(s,\vt)$. By~Lemma~\ref{lem:perturbationbound-norm} and~\cite[Theorem~2.6, estimation of~$P_L$]{Kloeckner2019}, we have
  \begin{equation}
    \norm{f_{s,\vt}-\xi}_\infty \ll_\eps \abs{s-2} + \nvt + \nvt^{\alpha_0-\eps}.\label{eq:bound-fst-xi}
  \end{equation}
  On the other hand, differentiating the eigenvalue equation~$\H_{2,0}[g_{s,\vt} f] = \lambda(s, \vt) f_{s,\vt}$ and integrating with respect to the Lebesgue measure, we get
  $$ \partial_{10}\lambda(s,\vt) \int_{[0,1]} f_{s,\vt}\df\nu = \int_{[0,1]}\Big((\log \pdt)g_{s,\vt} f_{s,\vt} + (g_{s,\vt}-\lambda(s,\vt))\partial_{10}f_{s,\vt}\Big)\df\nu. $$
  Here we recall the notation~\eqref{eq:def-pdt}. Setting~$(s,\vt)=(2,{\bf 0})$, with~$f_{2,\mathbf{0}}=\xi$ and~$g_{2,\mathbf{0}}=1$, gives~$\partial_{10}\lambda(2,\mathbf{0})=-\fd$. Using the bound~\eqref{eq:bound-fst-xi}, we get the approximation~\eqref{eq:approx-partial10lambda}.
\end{proof}

\begin{lemma}
  For all~$\eta>0$, there exists~$t_0>0$ and a unique function~$s_0:[-t_0,t_0]^d\to\C$ such that~$s_0(0) = 2$ and, for~$\nvt\leq t_0$,
  $$ \abs{s_0(\vt) - 2} \leq \eta, \qquad \lambda(s_0(\vt), \vt) = 1. $$
\end{lemma}
\begin{proof}
  This follows from a general form of the implicit functions theorem, \textit{e.g.}~\cite[Theorem~1.1]{Kumagai1980}, whose hypotheses are satisfied by virtue of Lemma~\ref{lem:approx-partialder}.
\end{proof}

In what follows we will not discuss the regularity of~$s_0(\vt)$ at each~$\vt$: we are only interested about its asymptotic behaviour around~$\vt=0$. We will use results on effective perturbation theory of linear operators, which have been worked out recently in~\cite{Kloeckner2019}.

\subsection{The sub-CLT case}

We first focus on the case where we do not aim at extracting a term of order~$2$ in the asymptotic expansion.

\begin{lemma}\label{lem:estim-s0-subclt}
  For~$\nvt \leq t_0$, we have
  $$ s_0(\vt) - 2 = \frac{1}{\fd}\int_0^1 (\e^{i\langle{\vt,\vphi(x)}\rangle}-1)\xi(x)\df x + O_\eps(\nvt^2 + \nvt^{2\alpha_0-\eps}). $$
\end{lemma}
\begin{proof}
  By Theorem~2.6 of~\cite{Kloeckner2019}, we have
  \begin{align*}
    \lambda(s, \vt) = {}& \int_0^1 \Pi_{s,\vt}\xi(x) \df x + O(\norm{\Pi_{s,\vt}-\Pi_{2,0}}^2) \\
    = {}& (s-2)\int_0^1 \big[\tfrac{\partial}{\partial s}\Pi_{s,0}\xi\big]_{s=2}(x)\df x + \int_0^1 \Pi_{2,\vt}\xi(x) \df x + O_\eps(\abs{s-2}^2 + \nvt^2 + \nvt^{2\alpha_0-\eps}) \\
    = {}& (s-2)\int_0^1 \big[\tfrac{\partial}{\partial s}\H^m[g_{s,0}\xi]\big]_{s=2}(x)\df x + \int_0^1 \H^m[\e^{i\langle \vt, \vphi \rangle }\xi](x) \df x + O_\eps(\abs{s-2}^2 + \nvt^2 + \nvt^{2\alpha_0-\eps}) \numberthis\label{eq:lambdast-prebound}
  \end{align*}
  by~\eqref{eq:def-pdt}. Since~$\int \H[f]\df\nu = \int f\df\nu$ and~$\H[\xi] = \xi$, the first integral is~$m\int_0^1 \log(x)\xi(x)\df x = -\fd$. The second equals~$\int_0^1 \e^{i\langle \vt, \vphi(x) \rangle}\xi(x)\df x$. For~$\alpha = \min(\alpha_0, 1)$, we have
  \begin{equation}
    \abs{\int_0^1 (\e^{i\langle \vt, \vphi \rangle}-1)\xi\df\nu} \leq \nvt^\alpha \int_0^1 \norm{\vphi}^\alpha \xi\df\nu = \nvt^\alpha \int_0^1\H[\norm{\vphi}^\alpha \xi] \df \nu \ll \nvt^\alpha\label{eq:lambdast-trivintegral}
  \end{equation}
  by our hypothesis~\eqref{eq:bound-norm-rough} and the triangle inequality (\emph{cf.}~\eqref{eq:triangleineq-hypoth-phi}). Setting~$s = s_0(\vt)$, we deduce~$s_0(\vt)-2 = O(\nvt + \nvt^{\alpha_0})$ by combining~\eqref{eq:lambdast-prebound} and \eqref{eq:lambdast-trivintegral}. Then another use of~\eqref{eq:lambdast-prebound} yields our claimed estimate.
\end{proof}

\subsection{The CLT case}

Our next goal is to extract the term of order~$\|\vt\|^2$ in the analysis above. We assume throughout that~$\alpha_0 > 1$.
In order to describe the order~2 coefficients, we introduce the following notation. Recalling~\eqref{eq:def-pdt}, let
\begin{align*}
  \mu_\vphi := {}& \frac{1}{\fd}\int_0^1 \vphi(x) \xi(x) \df x, \\
  \K[f] := {}& \frac1\xi (\Id - \H^m)^{-1} \N^m[f\xi], \numberthis\label{eq:def-opK} \\
  \vpsi := {}& \vphi + \mu_\vphi \log \pdt, \\
  \vchi := {}& \K \vpsi. \numberthis\label{eq:def-vchi}
\end{align*}
The well-definedness of~$\mu_\vphi$ follows from our hypothesis~$\alpha_0>1$. Note that~$\vchi$ is bounded on~$[0, 1]$, because
\begin{align*}
  \|\N^m[\vpsi \xi]\|_\infty \ll {}& \int_{[0, 1]} \|\vpsi\| \df \nu + \|\H^m[\vpsi\xi]\|_\infty \\
  \ll {}& 1 + \int_{[0, 1]} \|\vphi\|\df\nu + \sum_{h\in\cH^m} |h'(0)| \|\vphi|_{h(\I)}\|_{\infty},
\end{align*}
which is finite by~\eqref{eq:bound-norm-rough}.

\begin{lemma}\label{lem:estim-s0}
  If~$\alpha_0>1$, then
  \begin{equation}
    s_0(\vt) - 2 = \frac{1}{\fd}\int_0^1(\e^{i\langle \vt, \vphi(x) \rangle}-1)\xi(x)\df x - \vt^T C_\vphi \vt + O_\eps(\nvt^3 + \nvt^{\alpha_0+1-\eps}),\label{eq:approxs0-alpha-1-2}
  \end{equation}
  with
  \begin{equation}
    C_\vphi = \frac{1}{\fd}\int_0^1\Big(\frac12(\vpsi-\vphi)\cdot(\vpsi-\vphi)^T + \vphi\cdot(\vpsi-\vphi)^T + \vpsi\cdot\vchi^T \Big)\xi\df\nu. \label{eq:def-Cgam}
  \end{equation}
  Moreover, if~$\alpha_0> 2$, then
  \begin{equation}
    s_0(\vt) - 2 =  i \langle\vt, \mu_\vphi\rangle - \tfrac12 \vt^T \Sigma_\vphi \vt + O_\eps(\nvt^3 + \nvt^{\alpha_0-\eps}),\label{eq:approxs0-alpha-2}
  \end{equation}
  with
  \begin{equation}
    \Sigma_\vphi := \frac{1}{\fd}\int_0^1 (\vpsi + \vchi - \vchi\circ T^m\big)\cdot(\vpsi + \vchi - \vchi\circ T^m\big)^T \xi\df\nu.\label{eq:def-sigmagam}
  \end{equation}
\end{lemma}

\begin{remark}
  With the definition~\eqref{eq:def-sigmagam}, it is clear that the matrix~$\Sigma_\vphi$ is symmetric, positive semi-definite. It is definite if and only if the vectors~$\{(\vpsi + \vchi - \vchi\circ T^m)(x), x\in(0,1)\}$ span the whole space~$\R^d$. By our hypothesis~\eqref{eq:bound-norm-rough}, we have~$\int_0^1 \|\vphi\| \abs{\log\pdt} \df\nu < \infty$ whenever~$\alpha_0>1$, so that the matrix~$C_\vphi$ is well-defined in this case. Similarly, the matrix~$\Sigma_\vphi$ is well-defined whenever~$\alpha_0>2$, since in that case~$\int_{[0, 1]} \|\vpsi \cdot \vpsi^T\|\df\nu \ll 1 + \int_{[0, 1]} \|\vphi\|^2 \df\nu < \infty$.
\end{remark}

\begin{proof}
  We extend the computations of Lemma~\ref{lem:estim-s0-subclt}, using our hypothesis on~$\alpha_0$ to expand the quantity~$g_{s,\vt} = \e^{i\langle \vt, \vphi\rangle} \pdt^{s-2}$ to order~$2$ at~$s=2$ and order~$1$ at~$\vt=0$. Let
  $$ \Delta_{s,\vt} = \Pi_{s,\vt} - \Pi_{2,0}. $$
  We write~$\e^{i\langle \vt, \vphi \rangle} = 1 + i\langle \vt, \vphi \rangle + O((\norm{\vphi}\nvt)^{\min(2,\alpha_0-\eps)})$. Letting~$s = s_0(\vt) = 2 + O(\nvt)$, we obtain
  \begin{align*}
    \int_0^1 \Delta_{s,\vt}\xi\df\nu = {}& \int_0^1 (\e^{i\langle \vt, \vphi \rangle}-1)\xi\df\nu - \fd(s-2) + i(s-2) \int_0^1 \langle \vt, \vphi \rangle (\log \pdt)\xi\df\nu \\
    {}& \qquad + \frac12(s-2)^2\int_0^1(\log\pdt)^2\xi\df\nu +  O(\nvt^3 + \nvt^{1+\alpha_0-\eps}) \\
    = {}&  \int_0^1 (\e^{i\langle \vt, \vphi \rangle}-1)\xi\df\nu - \fd(s-2) - \vt^T C_{1,\vphi} \vt + O(\nvt^3 + \nvt^{1+\alpha_0-\eps}),
  \end{align*}
  where~$C_{1,\vphi} := \int_0^1 \vphi \cdot (\vpsi-\vphi)^T \xi\df\nu + \frac12\int_0^1 (\vpsi-\vphi)\cdot(\vpsi-\vphi)^T \xi\df\nu$.
  We use again Theorem~2.6 of~\cite{Kloeckner2019}, getting
  $$ \lambda(s, \vt) = 1 + \int_0^1 \Delta_{s,\vt}\xi \df\nu + \int_{[0,1]} \Delta_{s,\vt} (\Id-\H^m)^{-1}(\Id - \P) \Delta_{s,\vt}\xi\df\nu + O(\norm{\Delta_{s,\vt}}^3). $$
  By computations similar to~\eqref{eq:bound-normdiff-Pi-t}, we have
  $$ \norm{\Delta_{s,\vt}[f] - i\H^m[\langle \vt, \vphi \rangle f] - i\fd^{-1}\Big(\int_0^1\langle \vt, \vphi \rangle \xi\df\nu\Big)\H^m[\log\pdt f]}_0 \ll_\eps (\nvt^{\alpha_0-\eps} + \nvt^{2})\norm{f}_0. $$
  Note that the left-hand side can be written~$\norm{\Delta_{s,t}[f] - i\langle{\vt, \H^m[\vpsi f]}\rangle}_0$. We deduce
  \begin{align*}
    \int_{[0,1]} \Delta_{s,\vt} & (\Id-\H^m)^{-1}(\Id - \P) \Delta_{s,\vt}\xi\df\nu  \\
    = {}& -\int_0^1 \langle \vt, \vpsi \rangle (\Id-\H^m)^{-1}(\Id-\P)\H^m[\langle \vt, \vpsi \rangle \xi]\df\nu + O_\eps(\nvt^{3} + \nvt^{1+\alpha_0-\eps}).    
  \end{align*}
  Since~$\P\H^m = \P$ and~$\H^m = \P + \N^m$, we have~$(\Id-\H^m)^{-1}(\Id-\P)\H^m[f\xi] = \xi\K[f]$ where we recall the definition~\eqref{eq:def-opK}. Therefore, we have
  $$ -\int_0^1 \langle \vt, \vpsi \rangle (\Id-\H^m)^{-1}(\Id-\P)\H^m[\langle \vt, \vpsi \rangle \xi]\df\nu = -\vt^T C_{2,\vphi} \vt, $$
  where~${\bf u}^T$ denotes the transpose of the column vector~${\bf u}$, and~$C_{2,\vphi} = \int_{[0,1]} \big( \vpsi \cdot \K[\vpsi]^T \big)\xi\df\nu$. This proves~\eqref{eq:approxs0-alpha-1-2} with~$C_\vphi = C_{1,\vphi} + C_{2,\vphi}$ as claimed.

  To prove~\eqref{eq:approxs0-alpha-2}, we note that by the hypothesis~$\alpha_0>2$, the quantity~\eqref{eq:def-sigmagam} is well-defined. In order to expand it, we first note that, with the definition~\eqref{eq:def-vchi}, that by construction
  $$ \P[\vpsi \xi] = \P[\vphi \xi + \mu_{\vphi} \xi \log \pdt] = \Big(\int \vphi \xi + \fd^{-1}\Big(\int\vphi \xi\df\nu\Big) \Big(\int(\log\pdt)\xi\df\nu\Big)\Big)\xi \df\nu= 0 $$
  since~$\int(\log\pdt) \xi = -\fd$. Therefore
  $$ \vchi\xi = (\Id-\H^m)^{-1}\N^m[\vpsi\xi] = \sum_{j\geq 1}\H^{jm}[\vpsi\xi] = \H^m[(\vpsi+\vchi)\xi]. $$
  By the property~$\int f (g\circ T^m)\df\nu = \int \H^m[f] g\df\nu$, we have
  $$ \int (\vpsi + \vchi)\cdot(\vchi \circ T^m)^T \xi\df\nu = \int \H^m[(\vpsi+\vchi)\xi] \cdot \vchi^T \df\nu= \int \vchi \cdot \vchi^T \xi\df\nu. $$
  Similarly, we have
  $$ \int (\vchi \circ T^m) \cdot (\vchi \circ T^m)^T \xi\df\nu = \int \H^m[(\vchi\circ T^m)\xi] \cdot \vchi^T\df\nu = \int \vchi \cdot \vchi^T \xi\df\nu. $$
  We deduce that, with the definition~\eqref{eq:def-sigmagam}, we have
  \begin{align*}
    \Sigma_\vphi = {}& \int (\vpsi + \vchi)\cdot (\vpsi + \vchi)^T \xi\df\nu - 2 \int (\vpsi + \vchi)(\vchi \circ T^m)\df\nu + \int (\vchi \circ T^m) \cdot (\vchi \circ T^m)^T \xi\df\nu \\
    = {}& \int \vpsi \cdot \vpsi^T \xi\df\nu + 2 \int \vpsi \cdot \vchi^T\xi\df\nu.
  \end{align*}
  On the other hand, expanding the squares in~\eqref{eq:def-Cgam}, we find
  $$ 2C_\vphi = \int \vpsi\cdot\vpsi^T\xi\df\nu - \int\vphi\cdot\vphi^T\xi\df\nu + 2\int \vpsi\cdot\vchi^T\xi\df\nu. $$
  The claimed formula~\eqref{eq:approxs0-alpha-2} follows by the Taylor expansion~$\e^{iu} = 1 + u + \frac12u^2 + O(u^{\alpha_0-\eps})$ with~$u = i \langle \vt, \vphi \rangle$.
\end{proof}

\section{Proof of Theorem~\ref{th:main-general}}

Recall that~$\Omega_Q$ consists of the rationals in~$(0,1]$ of denominators at most~$Q$, and let
\begin{equation}
  \chi_Q(\vt) := \sum_{x\in\Omega_Q} \exp(\langle{i\vt, S_\vphi(x)}\rangle).\label{eq:def-chiQ}
\end{equation}

\begin{proposition}\label{prop:asymp-E-s0}
  For all~$\eps>0$, there exists~$\delta, t_0>0$ such that for~$\nvt\leq t_0$, we have
  $$ \E_Q(\e^{i\langle{\vt,S_\vphi(x)}\rangle}) = Q^{s_0(\vt)-2} \big\{1 + O_\eps(Q^{-\delta} + \nvt + \nvt^{\alpha_0-\eps})\big\}. $$
\end{proposition}
\begin{proof}
  Recall that~$\chi_Q(\vt)$ was defined in~\eqref{eq:def-chiQ}, so that~$\E_Q(\e^{i\langle{\vt,S_\vphi(x)}\rangle}) = \chi_Q(\vt) / \chi_Q(0)$. Let~$\Omega\geq1$ be a parameter, and~$w:\R_+\to[0,1]$ be a smooth function satisfying
  \begin{equation}
    \1_{[0,1]} \leq w \leq \1_{[0,1+\Omega^{-1}]}, \qquad \|w^{(j)}\|_\infty \ll_j \Omega^j. \label{eq:smoothing-bounds}
  \end{equation}
  Then by a trivial bound on the contribution of~$q\in[Q, Q(1+\Omega^{-1})]$, we have
  \begin{equation}
    \chi_Q(\vt) = O(\Omega^{-1}Q^2) + {\tilde \chi}_Q(\vt), \qquad \text{where} \quad {\tilde \chi}_Q(\vt) := \ssum{1 \leq a \leq q \\ (a,q)=1} \e^{i\langle{\vt,S_\vphi(a/q)}\rangle}w\Big(\frac qQ\Big).\label{eq:chiQ-smoothing}
  \end{equation}
  By Perron's formula, we have
  $$ {\tilde \chi}_Q(\vt) = \frac1{2\pi i}\int_{3-i\infty}^{3+i\infty} Q^s \gS(s,\vt) {\hat w}(s)\df s, $$
  where the Mellin transform~${\hat w}(s) = \int_0^\infty w(u) u^{s-1}\df u$ is defined for~$\Re(s)>0$.
  We move the contour to the line~$\Re(s) = 2-\delta$. If~$t_0, \delta>0$ are small enough, then by Proposition~\ref{prop:mero-continuation}, we encounter exactly one pole, at~$s = s_0(\vt)$. By Cauchy's theorem, we deduce
  \begin{equation}
    {\tilde \chi}_Q(\vt) = \underset{s=s_0(\vt)}\Res(Q^s \gS(s, \vt) {\hat w}(s)) + \frac1{2\pi i}\int_{2-\delta-i\infty}^{2-\delta+i\infty} Q^s \gS(s,\vt) {\hat w}(s)\df s.\label{eq:chiQ-contourshift}
  \end{equation}
  For some absolute constant~$C\geq 1$, Proposition~\ref{prop:mero-continuation} yields the bound
  $$ \abs{\gS(s,\vt)} \ll (\abs{\tau}+1)^{C}. $$
  On the other hand, we have~$\abs{{\hat w}(s)} \ll_C \Omega^{C+2} \abs{s}^{-C-2}$ for~$\Re(s)\in[1/2,3]$ by integration by parts and~\eqref{eq:smoothing-bounds}, so that
  \begin{equation}\label{eq:bound-restcontour}
    \abs{\int_{2-\delta-i\infty}^{2-\delta+i\infty} Q^s \gS(s,\vt) {\hat w}(s)\df s} \ll \Omega^{C+2} Q^{2-\delta}.
  \end{equation}
  Finally, we have by~\eqref{eq:extractpole-S}
  \begin{equation}
    \underset{s=s_0(\vt)}\Res( Q^s \gS(s,\vt) {\hat w}(s)) = -\frac{Q^{s_0} {\hat w}(s_0)}{\partial_{10}\lambda(s_0,\vt)} \sum_{0\leq j < m} \Pi_{s_0,\vt}^{(j)}\P_{s_0,\vt}[\1](1), \label{eq:chiQ-residue}
  \end{equation}
  where we abbreviated~$s_0 = s_0(\vt)$ in the right-hand side. Again by~\cite[Theorem~1.6, estimation of~$P_L$]{Kloeckner2019} and Lemma~\ref{lem:perturbationbound-norm}, we have~$\|P_{s_0,\vt} - \P_{2, 0}\|_0 \ll \nvt + \nvt^{\alpha_0-\eps}$, and therefore
  $$ \Pi_{s_0,\vt}^{(j)}\P_{s_0,\vt}[\1](1) = \frac1{2\log 2}\{1 + O_\eps(\nvt + \nvt^{\alpha_0 - \eps})\} $$
  for~$0\leq j < m$. The quantity~$\partial_{10}\lambda(s_0(\vt), \vt)$ was estimated in~\eqref{eq:approx-partial10lambda}. Finally, we have by~\eqref{eq:smoothing-bounds}
  $$ {\hat w}(s_0) = \frac12\{1+O_\eps(\Omega^{-1} + \nvt + \nvt^{\alpha_0-\eps})\}. $$
  Inserting these estimates in~\eqref{eq:chiQ-residue}, we deduce
  $$ \underset{s=s_0(\vt)}\Res( Q^s \gS(s,\vt) {\hat w}(s)) = \frac 3{\pi^2} Q^{s_0(\vt)} \big\{ 1 + O_\eps(\Omega^{-1} + \nvt + \nvt^{\alpha_0-\eps})\big\}. $$
  Grouping this~\eqref{eq:chiQ-smoothing}\eqref{eq:chiQ-contourshift} and \eqref{eq:bound-restcontour}, we conclude
  $$ \chi_Q(\vt) = \frac 3{\pi^2} Q^{s_0(\vt)} \big\{ 1 + O_\eps(\Omega^{-1} + \Omega^{C+2} Q^{-\delta} + \nvt + \nvt^{\alpha_0-\eps})\big\}. $$
  Our claim follows by optimizing~$\Omega = Q^{\delta/(C+3)}$ and dividing by~$\chi_Q(0)$.
\end{proof}

\begin{proof}[Proof of Theorem~\ref{th:main-general}]
  We use Proposition~\ref{prop:asymp-E-s0} along with Lemmas~\ref{lem:estim-s0-subclt}--\ref{lem:estim-s0} and the value~\eqref{eq:def-fd}.
\end{proof}

\section{Applications}\label{sec:applications}

When using the Berry--Esseen inequality, we will require a separate treatment of very small values of~$t$ in order to handle the error term~$O(Q^{-\delta})$ in Theorem~\ref{th:main-general} (the argument described in~\cite[Remark~3.8]{Goueezel2015} is not readily adapted since part of this term originates from counting pairwise coprime numbers).

\begin{lemma}\label{lem:Caq-roughbound}
  Suppose that the function~$\vphi$ satisfies~\eqref{eq:bound-norm-rough} and \eqref{eq:bound-Ckappa-rough}. Then we have
  \begin{equation}
    \E_Q(\e^{i \langle\vt, S_\vphi(x)\rangle}) = 1 + O(\norm{\vt}^{\alpha_0/3}\log Q).\label{eq:Caq-roughbound-it0}
  \end{equation}
\end{lemma}
\begin{proof}
  For all~$n\geq 1$, define~$c(n) := \sup_{x\in(\frac 1{n+1}, \frac1n)} \norm{\vphi(x)}^{\alpha_0/3}$. Then with the terminology of~\cite[p.750]{Baladi.Hachemi2008}, the function~$c$ has strong moments to order~$3$. Setting~$a_j(x) = \floor{1/T^{j-1}(x)}$, we deduce~$\E_Q(\sum_{j=1}^r c(a_j)) \ll \log Q$ by~\cite[Remark~1.2]{Baladi.Hachemi2008}. We conclude by taking expectations in the bound $\abs{\e^{i\langle{\vt,S_\vphi(x)}\rangle}-1} = O(\sum_{j=1}^r (\norm{\vt} \norm{\vphi(T^{j-1}(x))})^{\alpha_0/3})$.
\end{proof}

\begin{proof}[Proof of Corollary~\ref{cor:CLT}]
  In~\eqref{eq:thmgen-secondterm}, we estimate the integral~$\II_\phi(t)$ using~\cite[Proposition~2.1]{appB}, which amounts here to integrating the Taylor expansion of~$\e^{it\phi(x)}$ at~$t=0$. From~\eqref{eq:def-sigmagam}, and with the notation~\eqref{eq:def-mean}, we obtain
  \begin{equation}
    U(t) = it \mu - \frac{t^2}2 \sigma^2 + O_\eps(\abs{t}^{\min(3, \alpha_0-\eps)}),\label{eq:estim-psi-CLT}  
  \end{equation}
  where, with~$\psi(x) = \phi(x) + \mu \log x$,
  $$ \sigma = \frac{12\log 2}{\pi^2}\int_0^1 (\psi(x) + \chi(x) - \chi(T(x)))^2 \frac{\df x}{1+x}. $$
  We recall that~$\chi$ is related to~$\psi$ by~\eqref{eq:def-vchi} (with~$m=1$). It is clear that~$\sigma\geq 0$. If~$\sigma=0$, then the integrand vanishes identically, and we would conclude that~$\phi = -\mu\log - \chi + \chi\circ T$, contradicting our hypothesis.  We use the Berry-Esseen theorem~\cite[theorem~II.7.16]{Tenenbaum2015a},~\cite[equation~XVI.(3.13)]{Feller1971} with~$T=c(\log Q)^{1/2}$ for some parameter~$c>0$ to be chosen. Let
  $$ f(\tau) := \E_Q\Big(\exp\Big(i\tau\frac{S_\phi(x) - \mu\log Q}{\sigma\sqrt{\log Q}}\Big)\Big), \qquad g(\tau) := \e^{-\tau^2/2}. $$
  When~$\abs{\tau}\leq Q^{-\delta}$, we use Lemma~\ref{lem:Caq-roughbound} and a Taylor bound, getting
  $$ f(\tau) = 1 + O(\abs{\tau}^{\alpha_0/3}(\log Q)^{1-\alpha_0/6} + \abs{\tau}(\log Q)^{1/2}). $$
  When~$Q^{-\delta}<\abs{\tau}\leq (\log Q)^\eps$, we use Theorem~\ref{th:main-general}, getting
  \begin{align*}
    f(\tau) ={}& \exp\Big\{-\frac{\tau^2}2 + O\Big(\frac{\abs{\tau}}{(\log Q)^{1/2}} + \frac{\abs{\tau}^{\alpha_0-\eps}}{(\log Q)^{\alpha_0/2-1-\eps/2}} + \frac1{Q^{\delta}}\Big)\Big\} \\
    = {}& \e^{-\tau^2/2}\Big(1 + O\Big(\frac{\abs{\tau}}{(\log Q)^{1/2}} + \frac{\abs{\tau}^{\alpha_0-\eps}}{(\log Q)^{\alpha_0/2-1-\eps/2}} + \frac1{Q^{\delta}}\Big)\Big).
  \end{align*}
  Finally, when~$(\log Q)^{\eps} < \abs{\tau} \leq T$, we get
  \begin{align*}
    f(\tau) ={}& \exp\Big\{-\frac{\tau^2}2\Big(1 + O\Big(\frac{1}{\abs{\tau}(\log Q)^{1/2}} + \frac{\abs{\tau}^{\alpha_0-2-\eps}}{(\log Q)^{\alpha_0/2-1-\eps/2}} + \frac1{\tau ^2Q^{\delta}}\Big)\Big\} \\
    = {}& \exp\Big\{-\frac{\tau^2}2(1 + O(c^{\alpha_0-2-\eps})\Big\}
  \end{align*}
  and so, for some small enough choice of~$c>0$, we have~$\abs{f(\tau)} \leq \e^{-\tau^2/3}$.

  We deduce
  \begin{align*}
    {}&\int_{-T}^T \abs{\frac{f(\tau)-g(\tau)}{\tau}}\df \tau \\
    \ll {}& \int_0^{Q^{-\delta}} (\tau^{\alpha_0/3-1} (\log Q)^{1-\alpha_0/6} + (\log Q)^{1/2} + \tau)\df t \\
    {}& + \int_{Q^{-\delta}}^{(\log Q)^\eps} \e^{-\tau^2/2}((\log Q)^{-1/2} + \tau^{\alpha_0-1-\eps}(\log Q)^{-\alpha/2+1+\eps/2} + \tau^{-1} Q^{-\delta})\df \tau \\
    {}& + \int_{(\log Q)^\eps}^T \tau^{-1} \e^{-\tau^2/3} \df \tau \\
    \ll{}& (\log Q)^{-1/2} + (\log Q)^{-\alpha_0/2+1+\eps},
  \end{align*}
  and therefore by the Berry-Esseen inequality
  $$ \sup_{v\in \R}\Big| \P_Q\Big(\frac{S_\phi(x) - \mu\log Q}{\sigma\sqrt{\log Q}} \leq v\Big) - \Phi(v)\Big| \ll \frac1T + \int_{-T}^T \abs{\frac{f(\tau)-g(\tau)}\tau} \df \tau \ll \frac1{(\log Q)^{\min(1/2, \alpha_0/2-1-\eps)}} $$
  as claimed.
\end{proof}

\subsection{Central modular symbols}\label{sec:mod-symb}

Let $f(z) = \sum_{n\geq1}a_n\e(nz)$ be a non-zero primitive Hecke eigencuspform of weight $k$ for~$SL(2,\Z)$ with trivial multiplier. Note that~$k$ is necessarily even and~$k\geq 12$.

Define, for all integer~$1\leq m \leq k-1$ and all~$x\in\Q$, the modular symbol
\begin{equation}
  \syb{x}_{f,m} := \frac{(2\pi i)^m}{(m-1)!} \int_x^{i\infty} f(z) (z-x)^{m-1}\df z.\label{eq:def-syb}
\end{equation}

\begin{lemma}\label{lem:symbol-holder}
  For~$m>k/2$, the function~$x\mapsto \syb{x}_{f,m}$, initially defined over~$\Q$, can be extended to a bounded function in~$\Hol^{1-\eps}(\R)$ for any~$\eps\in(0,1)$.
\end{lemma}
\begin{proof}
  By Deligne's bound~\cite{Deligne1974}, we have~$\abs{a_n}\ll_{\eps,f} n^{(k-1)/2+\eps}$. Therefore the sum~$\sum_{n\geq 1} \abs{a_n}/n^m$ is finite, and we deduce by Fubini's theorem that for~$x\in\Q$,
  $$ \syb{x}_{f,m} = (-1)^m \sum_{n\geq 1} \frac{a_n \e(nx)}{n^m}, $$
  and the left-hand side is now defined for~$x\in\R$. By~\cite[Theorem~5.3]{Iwaniec1997}, we have
  \begin{equation}
    \sum_{n\leq t} a_n \e(nx) \ll_f t^{k/2} \log(2t) \qquad (t\geq 1)\label{eq:estim-an-exp}
  \end{equation}
  uniformly in~$x\in\R$. Let~$x, x'\in\R$,~$\delta = \abs{x-x'}$, and for~$t\geq 1$, $S(t) := \sum_{n\leq t} a_n (\e(nx)-\e(nx'))$. Then using~\eqref{eq:estim-an-exp} and partial summation, we obtain~$\abs{S(t)} \ll_{\eps,f} t^{k/2+\eps}\min(1,\delta t)$, and so
  \begin{align*}
    \abs{\syb{x}_{f,m} - \syb{x'}_{f,m}} \ll_{\eps,f} {}& \int_1^\infty t^{-m-1+k/2+\eps}\min(1,\delta t) \df t \\
    \ll_{\eps,f} {}& \delta + \delta^{m-k/2-\eps}
  \end{align*}
  as claimed, since~$m\geq k/2+1$.
\end{proof}

\begin{lemma}\label{lem:recip-symbol}
  For any~$\eps\in(0,1)$, some function~$\phi_f \in \Hol^{1-\eps}([0,1], \C)$, and all~$x\in\Q\smallsetminus\{0\}$, we have
  $$ \syb{x}_{f,k/2} = \syb{\tfrac{-1}x}_{f,k/2} + \phi_f(x). $$
  Moreover, we have~$\HN{1-\eps}{\phi_f\circ h} \ll 1$ uniformly for~$h\in\cH^\ast$.
\end{lemma}

We will actually only require the last bound for~$h\in\cH$.

\begin{proof}
  For~$\Im(z)>0$, define
  \begin{align*}
    \hat f(z):= {}& \sum_{n\geq 1}\frac{a_n}{n^{k-1}}\e(nz) \\
    = {}& \sum_{n\geq 1}a_n \e(nz) \frac{(-2\pi i)^{k-1}}{(k-1)!}\int_0^{i\infty}\tau^{k-2}\e^{2\pi i n \tau}\df\tau \\
    = {}& \frac{(-2\pi i)^{k-1}}{(k-1)!}\int_z^{i\infty}(\tau-z)^{k-2}f(\tau)\df\tau.
  \end{align*}
  Using the modularity relation $f(-1/z)=z^{k} f(z)$, and changing variables $\tau\to-1/\tau$, we obtain
  \begin{align*}
    \frac{(k-1)!}{(-2\pi i)^{k-1}} z^{k-2}\hat f(-1/z)&=
    z^{k-2}\int_{-1/z}^{i\infty}(\tau+1/z)^{k-2}f(\tau)\df\tau=\int_{z}^{0}(-1/\tau+1/z)^{k-2}(z\tau)^{k-2}f(\tau)\df\tau\\
    &=-\int_{0}^{z}(\tau-z)^{k-2}f(\tau)\df\tau
  \end{align*}
  and so the period polynomial of~$f$
  \begin{equation*}
    r_f(z):=\hat f(z)-z^{k-2}\hat f(-1/z) = \frac{(-2\pi i)^{k-1}}{(k-1)!} \int_0^{i\infty} (\tau-z)^{k-2}f(\tau) \df\tau, \qquad (\Im(z)>0),
  \end{equation*}
  is indeed a polynomial in~$z$ of degree at most $k-2$.

  Let now~$x\in\Q_{>0}$. As~$\delta\to0$ with~$\delta>0$, we have
  \begin{align*}
    \hat f(x(1+i\delta)) = {}& \frac{(-2\pi i)^{k-1}}{(k-1)!} \int_{x(1+i\delta)}^{i\infty}(\tau-x-ix\delta)^{k-2}f(\tau)\df\tau\\
    = {}& \frac{(-2\pi i)^{k-1}}{(k-1)!} \bigg(\int_{x}^{i\infty}-\int_{x}^{x(1+i\delta)}\bigg) (\tau-x-ix\delta)^{k-2}f(\tau) \df\tau.
  \end{align*}
  The second integral is $O_{M, x, f}(\delta^{M})$ as $\delta\to0$ for any fixed $M>0$, since $f$ is a cusp form. Thus, by the binomial formula, we obtain
  \begin{align*}
    \hat f(x(1+i\delta))&=\frac{(-1)^{k-1}}{(k-1)} \sum_{\ell=0}^{k-2}\frac{(2\pi x\delta)^{\ell}}{\ell!} \syb{x}_{f,k-1-\ell} + o_{x, f}(\delta^{k-2}).
  \end{align*}
  In the same way, since $-\frac1{x(1+i\delta)}=-\frac1x(1-i\delta')$ with $\delta' = \delta/(1+i\delta)$, so that $\Re(\delta)>0$, we have
  \begin{align*}
    {}& (x(1+i\delta))^{k-2} \hat f(-\frac1{x(1+i\delta)}) \\
    {}& \quad = \frac{(-1)^{k-1}}{(k-1)} \sum_{\ell=0}^{k-2} (i\delta)^\ell \sum_{j=0}^{\ell} \frac{(-2\pi i)^j}{j!}\binom{k-2-j}{\ell-j}x^{k-2-j} \left\langle\frac{-1}x\right\rangle_{f,k-1-j} + o_{x, f}(\delta^{k-2}).
  \end{align*}
  With~$m := k/2-1$, reading the coefficients of~$\delta^{m}$ on each side of the definition of~$r_f(x(1+i\delta))$, and since~$k$ is even, we deduce
  $$ \syb{x}_{f,1+m} - \sum_{j=0}^{m} c_{j,k} x^j \left\langle\frac{-1}x\right\rangle_{f,1+m+j} = -(k-1)(2\pi)^{m+1} i^m r_f^{(m)}(x), $$
  $$ c_{j,k} := j!\binom{m}{j} \binom{m+j}{j} (-2\pi i)^{-j} . $$
  We single out the term~$j=0$. The function
  $$ \phi_f(x) := \sum_{j=1}^{m} c_{j,k}x^j \left\langle\frac{-1}x\right\rangle_{f,1+m+j} -(k-1)(2\pi)^{m+1} i^m r_f^{(m)}(x), $$
  defines, by Lemma~\ref{lem:symbol-holder}, a function in~$\Hol^{1-\eps}([\frac1{n+1}, \frac1n])$ for all~$\eps\in(0,1)$ and~$n\geq 1$. The value~$c_{0,k}=1$ proves our claimed formula. Finally, for~$h\in\cH$, by the rules~\eqref{eq:Ckappa-product}, \eqref{eq:Ckappa-composition}, \eqref{eq:Ckappa-morediff} and~$1$-periodicity of~$x\mapsto \syb{x}_{f,1+m+j}$, we have
  \begin{align*}
    \HN{1-\eps}{\phi_f\circ h} \ll_f {}& \sum_{j=1}^m \HN{1-\eps}{x\mapsto h(x)^j \syb{x}_{f,1+m+j}} + \HN{1-\eps}{r_f^{(m)}\circ h} \\
    \ll_f {}& \sum_{j=1}^m \big(\norm{(h^j)'}_\infty^{\eps}\norm{h^j}_\infty^{1-\eps} + \norm{h^j}_\infty\big) + \norm{h'}_\infty^{1-\eps} \\
    \ll_f {}& 1.
  \end{align*}
  By the rule~\eqref{eq:Ckappa-composition} again, and since~$\|h\|_{(1)} \leq \|h'\|_\infty \ll 1$ for~$h\in\cH^\ast$, we deduce that the same bound~$\HN{1-\eps}{\phi_f\circ h} \ll 1$ holds for~$h\in\cH^*$.
\end{proof}

\begin{proof}[Proof of Theorems~\ref{thm:modsymbol} and~\ref{thm:modsymbol-LD}]

  Iterating Lemma~\ref{lem:recip-symbol}, we have for all~$x\in\Q\cap (0, 1]$,
  \begin{equation}
    \syb{x}_{f,k/2} = \sum_{j=1}^{r}
    \phi((-1)^{j-1} T^{j-1}(x)) + \syb{0}_{f,k/2}.\label{eq:syb-iter}
  \end{equation}
  Note that changing the coordinates of the set~$\cR\subset \R^2$ by an amount~$O(1/\sqrt{\log Q})$ in~\eqref{eq:estim-modsymbol} does not alter the right-hand side, so that we may replace~$\syb{x}_{f,k/2}$ by~$\syb{x}_{f,k/2}-\syb{0}_{f,k/2}$. For all~$\vt\in\R^2$, identifying~$\C\simeq \R^2$ with basis~$(1, i)$, we let
  $$ \chi(t) := \E_Q\Big(\exp\Big\{\frac{i\langle \vt, \syb{x}_{f,k/2}-\syb{0}_{f,k/2}\rangle}{\sigma_f\sqrt{\log Q}}\Big\}\Big), $$
  where we recall that~$\sigma_f$ was defined in Theorem~\ref{thm:modsymbol}.
  We apply Theorem~\ref{th:main-general} with~$m=1$ and $d=2$. The hypothesis~\eqref{eq:bound-norm-rough} is satisfied with~$\alpha_0=4$, since the function $\phi_f$ is continuous on~$[0,1]$ by Lemma~\ref{lem:recip-symbol} and therefore bounded. The hypothesis~\eqref{eq:bound-Ckappa-rough} is satisfied for any~$\lambda_0 < \frac1{2-2\eps}$, by using Lemma~\ref{lem:recip-symbol} and noting that~$\HN{1-\eps}{\phi_f|_{h(\I)}} \ll \abs{h'(0)}^{-1+\eps} \HN{1-\eps}{\phi_f\circ h}$. Using~\eqref{eq:thmgen-secondterm} along with the expressions~\eqref{eq:approxs0-alpha-2}, \eqref{eq:def-sigmagam}, we obtain for some~$\mu_f \in\R^2$ and real~$2\times 2$ matrix~$\Sigma_f$ the estimate
  $$ \chi(\vt) = \exp\Big\{ i \langle \vt, \mu_f \rangle - \tfrac12 \vt^T \Sigma_f \vt + O\Big(\frac{\nvt + \nvt^3}{\sqrt{\log Q}} + \frac1{Q^\delta}\Big)\Big\}. $$
  To compute the variance, we appeal to the bound
  \begin{equation}
    \E_Q(|\syb{x}_{f,m} - \mu_f\log Q|^4) \ll (\log Q)^{2}.\label{eq:syb-3mom-bound}
  \end{equation}
  This can be proved by shifting to the setting of~\cite{BaladiVallee2005}, where the variable~$t$ is extended to a complex neighborhood of the origin; the functions~$U, V$ defined in~\eqref{eq:def-psi-1} are, in this case, analytic in~$t$ near the origin, by boundedness of~$\vphi_j$.
  Then by \emph{e.g.}~\cite[th.~25.12]{Billingsley1995} and~\eqref{eq:syb-3mom-bound}, we find
  $$ \mu_f = \lim_{Q\to \infty} \frac{\E_Q(\syb{x}_{f,k/2})}{\log Q}, \qquad \Sigma_f = \lim_{Q\to\infty} \frac{\E_Q(\syb{x}_{f,k/2} \syb{x}_{f,k/2}^T)}{\sigma_f^2\log Q}. $$
  On the other hand, as~$Q\to\infty$, we have the following asymptotic formulae, where now~$\syb{x}_{f,k/2}$ is interpreted as a complex number:
  \begin{equation}
    \E_Q(\syb{x}_{f,m}) = o(\log Q),\label{eq:syb-expec}
  \end{equation}
  \begin{equation}
    \E_Q(\syb{x}_{f,m}^2) = o(\log Q),\label{eq:syb-secmom-sign}
  \end{equation}
  \begin{equation}
    \E_Q(|\syb{x}_{f,m}|^2) \sim 2 \sigma_f^2 \log Q.\label{eq:syb-secmom-abs}
  \end{equation}
  These statements can be proven (in a stronger form) by standard methods, using orthogonality of additive characters, the approximate functional equation~\cite[Theorem~5.3]{IwaniecKowalski2004} and Rankin-Selberg theory~\cite[Chapter~13.6]{Iwaniec1997}. The value~$\sigma_f$ appears as~$\sigma_f^2 = \Res_{s=1} L(f\times \bar{f}, 1)$, which is evaluated in~\cite[eq.~(13.52)]{Iwaniec1997}. Note that the analogues of~\eqref{eq:syb-secmom-sign} and~\eqref{eq:syb-secmom-abs} with a single average over numerator have recently been computed in~\cite{BlomerEtAl2018}; in their result as stated, however, the denominator is assumed to be prime.

  The equality~\eqref{eq:syb-expec} shows that~$\mu_f = 0$. The equality~\eqref{eq:syb-secmom-sign} shows that the matrix~$\Sigma_f$ is a multiple of the identity, and the equality~\eqref{eq:syb-secmom-abs} then shows that~$\Sigma_f = \Id$. Using the Berry--Esseen inequality, along with Lemma~\ref{lem:Caq-roughbound}, concludes the proof of Theorem~\ref{thm:modsymbol}.

  To justify Theorem~\ref{thm:modsymbol-LD}, we use again the analyticity of~$U, V$ defined in~\eqref{eq:def-psi-1} in a neighborhood of the origin. Hypothesis~$(4)$ of~\cite{Hwang1996} is therefore satisfied in our case with~$\phi(n)$ replaced by~$\log Q$, and~\eqref{eq:modsym-ldp} follows by~\cite[Theorem~1]{Hwang1996}, taking~$t = \eps \sqrt{\log Q}$.
\end{proof}

\subsection{Central value of the Estermann function}

\begin{proof}[Proof of Theorem~\ref{thm:estermann}]
  The Estermann function~$D$ corresponds to~$D_0$ in notation of~\cite{Bettin2016}. For~$x\in\Q$, it is the analytic continuation of
  $$ D(s, x) = \sum_{n\geq 1} \frac{\e(nx) \tau(n)}{n^s}, $$
  initially defined for~$\Re(s) > 1$, evaluated at~$s=\tfrac12$. We recall that~$\tau(n)$ is the number of divisors of~$n$. We use~\cite[Lemma~10]{Bettin2016}, noting that the quantity~$v_{j-1}/v_j$ corresponds to~$T^{j-1}(x)$. Therefore
  \begin{align*}
    D(\tfrac12, x) = {}& \zeta(\tfrac12)^2 + \sum_{j=1}^r \phi_j(T^{j-1}(x)),
  \end{align*}
  where
  \begin{align*}
    \phi_j(x) = {}& \tfrac12 x^{-1/2}\big(\log(1/x) + \gamma_0 - \log(8\pi) - \tfrac\pi 2\big) \\
    {}& + (-1)^{j-1} i\tfrac12 x^{-1/2}\big(\log(1/x) + \gamma_0 - \log(8\pi) + \tfrac\pi 2\big) + \zeta\big(\tfrac12\big)^2 + \cE((-1)^jx),
  \end{align*}
  and~$\cE$, which corresponds to~$\cE(0,\cdot)$ in the notation of~\cite[p.~6900]{Bettin2016}, is bounded and continuous. By comparing the cases~$N=0$ and~$N=1$ of~\cite[eq.~(3.17)]{Bettin2016}, we have
  \begin{equation}
    \cE(x) = \cE_1(x) + \sum_{j\in\{1, 2\}} \frac{(-1)^j}{j\pi}\Gamma(\tfrac12+j)^2 \Big(\frac{x}{2\pi i}\Big)^j\big( D(\tfrac12+j, -1/x) + \zeta(\tfrac12+j)^2/j\big),\label{eq:rel-E-E1}
  \end{equation}
  where~$\cE_1 \in \CC^1([0, 1])$. For~$j\geq 1$, the function~$D(\frac12+j, \cdot)$
  belongs to~$\Hol^{1/2-\eps}([0, 1], \C)$ for all~$\eps\in(0, 1/2)$. We deduce that the right-hand side of~\eqref{eq:rel-E-E1} defines, for all~$n\geq 1$, a function in~$\Hol^{1/2-\eps}([\frac1{n+1},\frac1n], \C)$. By an argument identical to Lemma~\ref{lem:recip-symbol}, we also have the bound~$\HN{1/2-\eps}{\cE\circ h} \ll 1$ for all~$h\in\cH^2$, and similarly for~$x\mapsto \cE(-x)$. This validates the hypothesis~\eqref{eq:bound-Ckappa-rough} with any~$\lambda_0 < 1$.

  On the other hand, setting~$\alpha_0 = 2-\eps$, we have
  $$ \sum_{h\in\cH} \abs{h'(0)} \|\vphi_j|_{h(\I)}\|_\infty^{\alpha_0} \ll \sum_{n\geq 1} n^{-2} \abs{n^{1/2}\log(2n)}^{\alpha_0} < \infty, $$
  and so the hypothesis~\eqref{eq:bound-norm-rough} holds. We apply Theorem~\ref{th:main-general} with~$m=2$, $d=2$, and the above given value~$\alpha = 2-\eps$. 
  In the estimate~\eqref{eq:thmgen-secondterm} we evaluate the integrals by appealing to~\cite[Corollary~3.1]{appB}. We deduce that there is a constant~$\mu\in\C $ such that, letting~$\sigma = 1/\pi$ and
  $$ \chi(t) := \E_Q\Big(\exp\Big\{i\Re\Big(t\frac{D(\frac12,x)-\mu\log Q}{\sigma(\log Q)^{1/2}(\log\log Q)^{3/2}}\Big)\Big\}\Big), \qquad (t\in\C), $$
  we have
  \begin{equation}
    \chi(t) = \exp\Big\{ - \frac{\abs{t}^2}{2} + O\Big(\frac1{Q^\delta} + \frac{\abs{t} + \abs{t}^2}{(\log \log Q)^{1-\eps}} \Big) \Big\}.\label{eq:esterman-estim-chi}
  \end{equation}
  We then obtain, by the Berry-Essen inequalities and the bound~\eqref{eq:Caq-roughbound-it0} for small frequencies, the statement of Theorem~\ref{thm:estermann} up to the value of the expectation. We compute the expectation from the initial object, using the expression~\cite[eq. (3.2)]{IwaniecKowalski2004} for Ramanujan sums. For~$s\in\C$ with~$\Re(s)>1$ we have
  \begin{align*}
    \ssum{1\leq a \leq q \\ (a, q)=1} D(s, a/q) =
    \zeta(s)^2 \sum_{\ell\mid q} \mu\Big(\frac q\ell\Big)\frac{\tau(\ell)}{\ell^{s-1}} \prod_{p^\nu\| \ell} (1 - \tfrac{\nu}{\nu+1}p^{-s}),
  \end{align*}
  and so by analytic continuation
  \begin{align*}
    \ssum{1\leq a \leq q \\ (a, q)=1} D(\tfrac12, a/q) = \zeta(\tfrac12)^2 \sum_{\ell\mid q} \mu\Big(\frac q\ell\Big) \tau(\ell) \ell^{1/2} \prod_{p^\nu \| \ell} (1 - \tfrac{\nu}{\nu+1}p^{-1/2}) = O_\eps(q^{1/2+\eps}).
  \end{align*}
  We deduce that
  \begin{equation}
    \E_Q(D(\tfrac12, x)) \ll_{\eps} Q^{-1/2+\eps}.\label{eq:bound-exp-D0}
  \end{equation}
  On the other hand, by~\eqref{eq:esterman-estim-chi}, we have
  $$ \chi(t) = 1 + O_\eps(Q^{-\delta} + (\log\log Q)^{-1/2+\eps} t + t^2), $$
  whereas exanding the exponential as~$\e^{iu} = 1 + iu + O(u^{3/2})$ in the definition of~$\chi(t)$, by~\eqref{eq:bound-exp-D0}, we get
  $$ \chi(t) = 1 + it\mu(\log Q)^{1/2}(\log\log Q)^{-3/2} + O(Q^{-1/3} + \abs{t}^{3/2}((\log Q)^{1/2} + \E_Q(|D(\tfrac12, x)|^{3/2})). $$
  Using the trivial bound~$\abs{D(\tfrac12, x)} \ll \sum_{j=1}^{r(x)} a_j(x)^{2/3}$ and Hölder's inequality, we find
  \begin{align*}
    \E_Q(|D(\tfrac12, x)|^{3/2}) \ll {}& \E_Q\Big(\big(\sum_{j=1}^{r(x)}1\big)^{1/2}\big(\sum_{j=1}^{r(x)} a_j(x)\big)\Big) \\
    \ll {}& (\log Q)^{5/2}
  \end{align*}
  by the bound~$r(x) \ll \log(\denom(x)+1)$ and~\cite[Theorem]{YaoKnuth1975}. Setting~$t = Q^{-\delta/2}$, we obtain
  $$ \mu (\log Q)^{1/2}(\log\log Q)^{-3/2} = O_\eps((\log\log Q)^{-1/2+\eps}) $$
  and so~$\mu = 0$, by letting~$Q\to\infty$.
\end{proof}

\subsection{Large moments of continued fractions expansions}

For all~$\lambda\geq 0$ and~$x\in\Q\cap(0,1)$, we recall that~$\Sigma_\lambda(x)$ was defined in~\eqref{eq:def-Sigmalambda}.
% \begin{equation}
%   \Sigma_\lambda(x) := a_1^\lambda + \dotsb + a_r^\lambda, \qquad (x = [0; a_1, \dotsc, a_r], a_r>1). \label{eq:def-Clambda}
% \end{equation}
%
For~$0<\alpha<2$, define
\begin{equation*}
  c_\alpha = \bigg(\frac{\Gamma(1-\alpha)\cos(\frac{\pi\alpha}{2})}{\pi^2/12}\bigg)^{1/\alpha}%\label{eq:def-stable-calpha}
\end{equation*}
and by continuity~$c_1 = \frac{6}\pi$. Let
\begin{equation*}
  g_\alpha(x) :=  \left\{
    \begin{aligned}
      {}& \frac1{2\pi} \int_{-\infty}^\infty \e^{-itx - (c_\alpha \abs{t})^\alpha(1 - i\sgn(t)\tan(\frac{\pi \alpha}2))} \df t, {}& \qquad (\alpha\neq 1) \\
      {}& \frac1{2\pi} \int_{-\infty}^\infty \e^{-itx - c_1 \abs{t}(1 + i\frac{2}\pi \sgn(t)\log \abs{t})} \df t, {}& \qquad (\alpha= 1)
    \end{aligned}\right.
\end{equation*}
be the probability distribution function of a stable law~$S_{\alpha}(c_\alpha, 1, 0)$ (see~\cite{SamorodnitskyTaqqu1994}), and
$$ G_\alpha(v) := \int_{-\infty}^v  g_\alpha(x)\df x. $$
\begin{theorem}\label{th:m_lambda}
  Let~$\lambda\geq 0$ and~$v\in\R$, and for~$\lambda<1$ define~$\mu_\lambda = \frac{12}{\pi^2}\sum_{n\geq 1} n^\lambda\log(\frac{(n+1)^2}{n(n+2)})$.
  \begin{enumerate}
    \item If~$\lambda<1/2$, then with for some~$\sigma_\lambda>0$, we have
    \begin{equation}
      \P_Q\Big(\frac{\Sigma_{1/2}(x) - \mu_\lambda \log Q}{\sigma_\lambda\sqrt{\log Q}} \leq v\Big) = \Phi(v) + O_\eps\Big(\frac1{(\log Q)^{\min(1/2, 1/(2\lambda)-1-\eps)}}\Big).\label{eq:moment-clt}
    \end{equation}
    \item If~$\lambda=1/2$, then with~$\sigma=(\pi^2/6)^{-1/2}$, we have
    \begin{equation}
      \P_Q\Big(\frac{\Sigma_{1/2}(x) - \mu_{1/2} \log Q}{\sigma\sqrt{\log Q\log \log Q}} \leq v\Big) = \Phi(v) + O\Big(\frac1{(\log\log Q)^{1-\eps}}\Big).\label{eq:moment-1/2}
    \end{equation}
    \item If~$1/2<\lambda<1$, then
    \begin{equation}\label{eq:moment-1/2-1}
      \P_Q\Big(\frac{\Sigma_\lambda(x) - \mu_\lambda \log Q}{(\log Q)^\lambda} \leq v\Big) = G_{1/\lambda}(v) + O\Big(\frac1{(\log\log Q)^{1-\eps}}\Big).
    \end{equation}
    \item If~$\lambda=1$, then letting~$\gamma_0$ denote the Euler constant,
    \begin{equation}
      \P_Q\Big(\frac{\Sigma_1(x)}{\log Q} - \frac{\log\log Q - \gamma_0}{\pi^2/12} \leq v\Big) = G_1(v) + O\Big(\frac1{(\log Q)^{1-\eps}}\Big).\label{eq:discrete-heinrich2}
    \end{equation}
    \item If~$\lambda>1$, then
    \begin{equation}
      \P_Q\Big(\frac{\Sigma_\lambda(x)}{(\log Q)^\lambda} \leq v\Big) = G_{1/\lambda}(v) + O\Big(\frac1{(\log\log Q)^{1-\eps}}\Big).\label{eq:moment-large}
    \end{equation}
  \end{enumerate}
  In all four cases the implied constant depends at most on~$\eps$ and~$\lambda$.
\end{theorem}

Except for~\eqref{eq:moment-1/2-1}, we expect the error terms to be optimal up to an exponent~$\eps$. The estimate~\eqref{eq:moment-large} is in accordance with results on the statistical distribution of~$\max_{1\leq j \leq r} a_j$~\cite{Hensley1991,CesarattoVallee2011}.

To prove Theorem~\ref{th:m_lambda}, we note that~$\Sigma_\lambda(x) = S_{\phi_\lambda}(x)$ with~$\phi_\lambda(x) := \floor{1/x}^\lambda$. The function~$\phi_\lambda$ satisfies the hypotheses of Theorem~\ref{th:main-general} with~$d=m=1$, $\alpha_0 = 1/\lambda-\eps$, $\kappa_0 = 1$, and all small enough exponents~$\lambda_0>0$.

\subsubsection{Case~$\lambda<1/2$}

We wish to apply Corollary~\ref{cor:CLT}. To proceed, we need to show that~$\phi_\lambda$ is not of the shape~$c\log + f - f\circ T$. Suppose that it were, let~$n\geq 1$ and~$x\in(0, 1)$ solve~$x = \frac1{n+x}$. Evaluating at~$x$ yields~$n^\lambda \sim -c\log n$ as~$n\to \infty$, a contradiction. Corollary~\ref{cor:CLT} may be applied and yields~\eqref{eq:moment-clt}.

\subsubsection{Case~$\lambda=1/2$}

The estimates~\eqref{eq:asympt-expect-UV}--\eqref{eq:thmgen-secondterm} hold, and the integral is evaluated in~\cite[Corollary~3.2]{appB}. With the notation~$\sigma = (\pi^2/6)^{-1/2}$ and
$$ \chi_{1/2, Q}(t) := \E_Q\Big(\exp\Big\{it\frac{\Sigma_{1/2}(x) - \mu_{1/2}\log Q}{\sigma\sqrt{\log Q \log\log Q}}\Big\}\Big), $$
we find that for~$\abs{t}\leq \log\log Q$,
$$ \chi_{1/2,Q}(t) = \exp\Big\{-\frac{3t^2}{\pi^2} + O\Big(\frac1{Q^\delta} + \frac{\abs{t}}{(\log Q)^{1/2-\eps}} + \frac{t^2}{(\log\log Q)^{1-\eps}}\Big)\Big\}.  $$
On the other hand, by~\eqref{eq:Caq-roughbound-it0}, we have~$\chi_{1/2,Q}(t) = 1 + O(\abs{t}^{1/2} \log Q)$. Inserting these two bounds in the Berry-Esseen theorem~\cite[theorem~II.7.16]{Tenenbaum2015a} yields the claimed conclusion~\eqref{eq:moment-1/2}.

\subsubsection{Case~$\lambda=1$}

We use the estimates~\eqref{eq:asympt-expect-UV}--\eqref{eq:def-psi-1}. Define
$$ \chi_{1, Q}(t) := \E_Q\Big(\exp\Big\{it\Big(\frac{\Sigma_{1}(x)}{\log Q} - \frac{\log\log Q-\gamma_0}{\pi^2/12}\Big)\Big\}\Big). $$
Then for~$0<t\leq (\log Q)^{1-\eps}$, we obtain by~\cite[Corollary~3.3]{appB}
$$ \chi_{1,Q}(t) = \exp\Big\{- \frac{it}{\pi^2/12}\Big(\abs{\log t} - \pi i\Big) + O\Big(\frac1{Q^\delta} + \frac{\abs{t}^{1-\eps} + \abs{t}^{2-\eps}}{(\log Q)^{1-\eps}}\Big)\Big\}, $$
and we may again conclude by the Berry-Esseen inequality.

\subsubsection{Case~$\lambda\not\in\{1/2, 1\}$}

Assume first~$\lambda>1$. Then we use the estimates~\eqref{eq:asympt-expect-UV}--\eqref{eq:def-psi-1}. Define
$$ \chi_{\lambda, Q}(t) := \E_Q\Big(\exp\Big\{it\frac{\Sigma_\lambda(x)}{(\log Q)^\lambda}\Big\}\Big). $$
Then, by~\cite[Corollary~3.2]{appB}, for~$0\leq t\leq \log\log Q$ we obtain
$$ \chi_{\lambda,Q}(t) = \exp\Big\{- (c_{1/\lambda} t)^{1/\lambda}(1 - i\tan(\tfrac{\pi}{2\lambda})) + O\Big(\frac1{Q^\delta} + \frac{t^{1/\lambda-\eps}}{(\log Q)^{1-\eps}} + \frac{t^{1/\lambda}}{(\log\log Q)^{1-\eps}}\Big)\Big\} $$
%where~$c_\ast= -\e(-1/(4\lambda))\Gamma(1-1/\lambda)/\log 2$. In particular, with the definition of~$c_{1/\lambda}$ given in~\eqref{eq:def-stable-calpha}, we find
%$$ \frac{c_\ast\log 2}{\pi^2/12} = - (c_{1/\lambda})^{1/\lambda}(1 - i\tan(\tfrac{\pi}{2\lambda})), $$
and we may again conclude by the Berry-Esseen inequality and Lemma~\ref{lem:Caq-roughbound}.

The case~$\lambda\in(1/2,1)$ follows by identical computations, the shift by~$\mu_\lambda\log Q$ being accounted for by the linear term in the asymptotic evaluation of~\eqref{eq:thmgen-secondterm}, as performed in~\cite[Corollary~3.2]{appB}
%the shift by~$\mu\log Q$ being accounted for by the term~$c_1 t$ in Lemma~\ref{lem:integralapplication-largemom}.

\subsection{Dedekind sums}

\begin{proof}[Proof of Theorem~\ref{th:dedekind}]
  By~\cite[Theorem~1]{Hickerson1977}, we have for~$x\in\Q\cap(0, 1)$ the equality
  $$ s(x) = \delta_{x} + \frac1{12}\sum_{j=1}^r \phi_j(T^{j-1}(x)), $$
  where~$\abs{\delta_x} \leq \frac{5}{12}$ and~$\phi_j(x) := (-1)^{j-1} \floor{1/x}$. Note that~$\phi_j$ depends only on the parity of~$j$. Since changing~$v$ by an amount~$O(1/\log Q)$ does not affect the right-hand side of~\eqref{eq:estim-dedekind}, we may replace~$s(x)$ by~$s(x)-\delta_x$. Let
  $$ \chi(t) :=  \E_Q\Big(\exp\Big\{it\frac{s(x)-\delta_{x}}{\log Q}\Big\}\Big). $$
  Then Theorem~\ref{th:main-general} applies with~$d=1$, $m=2$ and the functions~$\phi_j$ defined above, with~$\alpha_0 = 1-\eps$. We use the expression~\eqref{eq:def-psi-1} and refer to~\cite[Corollary~3.4]{appB} for the evaluation of the integral, obtaining
  $$ \chi(t) = \exp\Big\{ -\frac{\abs{t}}{2\pi} + O\Big(\frac1{Q^\delta} + \frac{\abs{t} + \abs{t}^{1-\eps}}{(\log Q)^{1-\eps}}\Big) \Big\}, $$
  and we conclude again by the Berry-Esseen inequality.
\end{proof}

\bibliographystyle{amsalpha2}
\bibliography{bib2}

\end{document}